\documentclass[10pt]{article}

\usepackage[utf8]{inputenc}

\usepackage{epsf}
\usepackage{amsmath}

\allowdisplaybreaks

\usepackage[showframe=false]{geometry}
\usepackage{changepage}

\usepackage{epsfig}
\usepackage{amssymb}

\usepackage{amsthm}
\usepackage{setspace}
\usepackage{cite}
\usepackage{mcite}

\usepackage{algorithmic}  
\usepackage{algorithm}

\usepackage{shadow}
\usepackage{fancybox}
\usepackage{fancyhdr}

\usepackage{color}
\usepackage[usenames,dvipsnames,svgnames,table]{xcolor}

\definecolor{mypurple}{rgb}{.4,.0,.5}

\usepackage[hyphens]{url}

\usepackage[colorlinks=true,
            linkcolor=black,
            urlcolor=blue,
            citecolor=purple]{hyperref}

\usepackage{breakurl}

\def\y{{\bf y}}

\def\x{{\bf x}}

\def\x{{\mathbf x}}

\def\u{{\bf u}}

\def\x{{\bf x}}
\def\y{{\bf y}}

\def\q{{\bf q}}
\def\m{{\bf m}}

\def\h{{\bf h}}

\def\be{\begin{equation}}
\def\ee{\end{equation}}
\def\ba{\left[\begin{array}}
\def\ea{\end{array}\right]}

\def\u{{\bf u}}

\def\x{{\bf x}}
\def\y{{\bf y}}

\def\q{{\bf q}}

\def\p{{\bf p}}

\def\1{{\bf 1}}

\def\0{{\bf 0}}

\def\calX{{\cal X}}
\def\calY{{\cal Y}}

\def\mR{{\mathbb R}}
\def\mN{{\mathbb N}}
\def\mE{{\mathbb E}}
\def\mS{{\mathbb S}}

\def\lp{\left (}
\def\rp{\right )}

\newtheorem{theorem}{Theorem}
\newtheorem{proposition}{Proposition}
\newtheorem{corollary}{Corollary}

\begin{document}

\begin{singlespace}

\title {Bilinearly indexed random processes -- \emph{stationarization} of fully lifted interpolation 
}
\author{
\textsc{Mihailo Stojnic
\footnote{e-mail: {\tt flatoyer@gmail.com}} }}
\date{}
\maketitle

\centerline{{\bf Abstract}} \vspace*{0.1in}

Our companion paper \cite{Stojnicnflgscompyx23}  introduced a very powerful \emph{fully lifted} (fl) statistical interpolating/comparison mechanism for bilinearly indexed random processes. Here, we present a particular realization of such fl mechanism that relies on a stationarization along the interpolating path concept. A collection of very fundamental relations among the interpolating parameters is uncovered, contextualized, and presented. As a nice bonus, in particular special cases, we show that the introduced machinery allows various simplifications to forms readily usable in practice. Given how many well known random structures and optimization problems critically rely on the results of the type considered here, the range of applications is pretty much unlimited. We briefly point to some of these opportunities as well.

\vspace*{0.25in} \noindent {\bf Index Terms: Random processes; comparison principles, lifting, stationarization}.

\end{singlespace}

\section{Introduction}
\label{sec:back}

Over the last half of a century, random processes comparison principles, as a main mathematical platform, enabled achieving some very famous results in a variety of scientific disciplines. Two of these principles clearly distinguished themselves by both simplicity and usefulness. They are, of course,
the well known Slepian's max \cite{Slep62} and the Gordon's minmax \cite{Gordon85} principle (more on the development, importance, and relevant prior and contemporary considerations can be found in, e.g.,  \cite{Sudakov71,Fernique74,Fernique75,Kahane86,Stojnicgscomp16,Adler90,Lifshits85,LedTal91,Tal05}). Circumventing them in studying plenty of random structures and optimization problems has been pretty much unimaginable over the last 20 years (see, e.g., \cite{Guerra03,Tal06,Pan10,Pan10a,Pan13,Pan13a,StojnicISIT2010binary,StojnicCSetam09,StojnicUpper10,StojnicICASSP10block,StojnicICASSP10var,StojnicICASSP10knownsupp,StojnicICASSP10knownsupp}). As a result of such studying, many, so-called, \emph{phase-transition} (PT) phenomena, which typically appear in random structures, have been either fully or almost fully accurately characterized.

Despite a significant progress made over the last two decades (see, e.g., \cite{Guerra03,Tal06,Pan10,Pan10a,Pan13,Pan13a}  for concepts  particularly related  to quadratic (or polynomial/tensorial) max type forms), characterizations of many PT phenomena remain out of reach (e.g., \cite{StojnicLiftStrSec13,StojnicMoreSophHopBnds10,StojnicRicBnds13,StojnicAsymmLittBnds11,StojnicGardSphNeg13,StojnicGardSphErr13} and references therein discuss various forms, including both quadratic and bilinear max and minmax ones, where satisfactory PT characterizations are still lacking). As discussed in \cite{Stojnicgscompyx16}, any further progress on this front seems to be directly related to the progress in studying and understanding the underlying random processes' comparisons. The mechanisms  of \cite{Stojnicgscomp16,Stojnicgscompyx16}  introduced the so-called \emph{partially lifting strategy} and made a strong step forward in that direction.
On the other hand, our companion paper \cite{Stojnicnflgscompyx23}  moved things to another level and introduced  a generic statistical interpolating/comparison mechanism called \emph{fully lifted} (fl). Here, we continue studying the fl  mechanism of \cite{Stojnicnflgscompyx23} and present its a particular realization that relies on a \emph{stationarization} along the interpolating path concept. We uncover, present, and discuss a set of very fundamental relations among the interpolating parameters that come out as direct consequences of the stationarization concept.

To make everything easier to follow, we, in Sections \ref{sec:gencon} discuss how the stationarization mechanism works on the so-called first level of full lifting. We then, in Section \ref{sec:rthlev},  proceed with the corresponding generalization that works for any $r$-th ($r\in\mN$), level of lifting.

\section{$\p,\q$-derivatives at the first level of full lifting}
\label{sec:gencon}

Let $\calX=\{\x^{(1)},\x^{(2)},\dots,\x^{(l)}\}$ with $\x^{(i)}\in \mR^n$ and $\calY=\{\y^{(1)},\y^{(2)},\dots,\y^{(l)}\}$ with $\y^{(i)}\in \mR^m$, vector $\p=[\p_0,\p_1,\p_2]$ with $\p_0\geq \p_1\geq \p_2= 0$, vector $\q=[\q_0,\q_1,\q_2]$ with $\q_0\geq \q_1\geq \q_2= 0$, and real parameters $\beta>0$ and $s$ be given and let function $f(\cdot)$ be
{\small\begin{equation}\label{eq:genanal1}
 f(G,u^{(4,1)},u^{(4,2)},\calX,\calY,\p,\q,\beta,s)= \frac{1}{\beta|s|\sqrt{n}} \log\lp \sum_{i_1=1}^{l}\lp\sum_{i_2=1}^{l}e^{\beta \lp (\y^{(i_2)})^T
 G\x^{(i_1)}+\|\x^{(i_1)}\|_2\|\y^{(i_2)}\|_2 (a_1u^{(4,1)}+a_2u^{(4,2)})\rp} \rp^{s}\rp.
\end{equation}}

\noindent In this paper, our primary focus is studying of this and similar functions in random mediums. More specifically, we consider $(m\times n)$ dimensional matrices  $G\in \mR^{m\times n}$ with i.i.d. standard normal components and independent (of $G$ and among themselves) standard normal random variables $u^{(4,1)}$ and $u^{(4,2)}$, weighted/scaled by $a_1=\sqrt{\p_0\q_0-\p_1\q_1}$ and $a_2=\sqrt{\p_1\q_1}$. For a scalar $\m=[\m_1]$, the following function turns out to be critically important for studying $f(G,u^{(4,1)},u^{(4,2)},\calX,\calY,\q,\beta,s)$
\begin{eqnarray}\label{eq:genanal2}
\xi(\calX,\calY,\p,\q,\m,\beta,s) &  \triangleq  &   \mE_{G,u^{(4,2)}}\frac{1}{\beta|s|\sqrt{n}\m_1} \nonumber \\
& & \times \log \mE_{u^{(4,1)}}\lp \sum_{i_1=1}^{l}\lp\sum_{i_2=1}^{l}e^{\beta \lp (\y^{(i_2)})^T
 G\x^{(i_1)}+\|\x^{(i_1)}\|_2\|\y^{(i_2)}\|_2 (a_1u^{(4,1)}+a_2u^{(4,2)})\rp} \rp^{s}\rp^{\m_1}.\nonumber \\
\end{eqnarray}
Throughout the paper, we adopt the convention that $\mE$ with a subscript denotes the expectation with respect to the subscript specified randomness. On the other hand, if the subscript of $\mE$ is not specified, the expectation is with respect to all underlying randomnesses. As in \cite{Stojnicnflgscompyx23}, we
find studying properties of $\xi(\calX,\calY,\p,\q,\m,\beta,s)$ by following into the footsteps of \cite{Stojnicgscomp16,Stojnicgscompyx16,Stojnicnflgscompyx23} as fairly convenient. To that end, we consider the following interpolating function $\psi(\cdot)$
\begin{equation}\label{eq:genanal3}
\psi(\calX,\calY,\p,\q,\m,\beta,s,t)  =  \mE_{G,u^{(4,2)},\u^{(2,2)},\h^{(2)}} \frac{1}{\beta|s|\sqrt{n}\m_1} \log \mE_{u^{(4,1)},\u^{(2,1)},\h^{(1)}} \lp \sum_{i_1=1}^{l}\lp\sum_{i_2=1}^{l}e^{\beta D_0^{(i_1,i_2)}} \rp^{s}\rp^{\m_1},
\end{equation}
where
\begin{eqnarray}\label{eq:genanal3a}
 D_0^{(i_1,i_2)} & \triangleq & \sqrt{t}(\y^{(i_2)})^T
 G\x^{(i_1)}+\sqrt{1-t}\|\x^{(i_1)}\|_2 (\y^{(i_2)})^T(b_1\u^{(2,1)}+b_2\u^{(2,2)})\nonumber \\
 & & +\sqrt{t}\|\x^{(i_1)}\|_2\|\y^{(i_2)}\|_2(a_1u^{(4,1)}+a_2u^{(4,2)}) +\sqrt{1-t}\|\y^{(i_2)}\|_2(c_1\h^{(1)}+c_2\h^{(2)})^T\x^{(i_1)}.
 \end{eqnarray}
In (\ref{eq:genanal3}), $\u^{(2,1)}$ and $\u^{(2,2)}$ are $m$ dimensional vectors of i.i.d. standard normals and $\h^{(1)}$ and $\h^{(2)}$ are $n$ dimensional vectors of i.i.d. standard normals. These four vectors are assumed to be independent among themselves and of all other random quantities. Reweighting/scaling factors satisfy $b_1=\sqrt{\p_0-\p_1}$ and $b_2=\sqrt{\p_1}$ and $c_1=\sqrt{\q_0-\q_1}$ and $c_2=\sqrt{\q_1}$. It is not that difficult to see that $\xi(\calX,\calY,\p,\q,\m,\beta,s)=\psi(\calX,\calY,\p,\q,\m,\beta,s,1)$ and since $\psi(\calX,\calY,\p,\q,\m,\beta,s,0)$ is typically easier to handle than $\psi(\calX,\calY,\p,\q,\m,\beta,s,1)$, establishing a connection between  $\psi(\calX,\calY,\p,\q,\m,\beta,s,1)$ and $\psi(\calX,\calY,\p,\q,\m,\beta,s,0)$ is basically equivalent to directly connecting $\xi(\calX,\calY,\p,\q,\m,\beta,s)$ to $\psi(\calX,\calY,\p,\q,\m,\beta,s,0)$. To make the exposition easier to follows, we below set
\begin{eqnarray}\label{eq:genanal4}
\u^{(i_1,1)} & =  & \frac{G\x^{(i_1)}}{\|\x^{(i_1)}\|_2} \nonumber \\
\u^{(i_1,3,1)} & =  & \frac{(\h^{(1)})^T\x^{(i_1)}}{\|\x^{(i_1)}\|_2} \nonumber \\
\u^{(i_1,3,2)} & =  & \frac{(\h^{(2)})^T\x^{(i_1)}}{\|\x^{(i_1)}\|_2}.
\end{eqnarray}
Denoting by  $G_{j,1:n}$  the $j$-th row of $G$ and by $\u_j^{(i_1,1)}$ the $j$-th component of $\u^{(i_1,1)}$, one from (\ref{eq:genanal4}) has
\begin{eqnarray}\label{eq:genanal5}
\u_j^{(i_1,1)} & =  & \frac{G_{j,1:n}\x^{(i_1)}}{\|\x^{(i_1)}\|_2},1\leq j\leq m.
\end{eqnarray}
For any fixed $i_1$, one trivially has that the elements of $\u^{(i_1,1)}$, $\u^{(2,1)}$, $\u^{(2,2)}$, $\u^{(i_1,3,1)}$, and $\u^{(i_1,3,2)}$ are i.i.d. standard normals. Setting ${\mathcal U}_k=\{u^{(4,k)},\u^{(2,k)},\h^{(k)}\},k\in\{1,2\}$, one then easily rewrites (\ref{eq:genanal3}) as
\begin{equation}\label{eq:genanal6}
\psi(\calX,\calY,\p,\q,\m,\beta,s,t)  =  \mE_{G,{\mathcal U}_2} \frac{1}{\beta|s|\sqrt{n}\m_1} \log \mE_{{\mathcal U}_1} \lp \sum_{i_1=1}^{l}\lp\sum_{i_2=1}^{l}A^{(i_1,i_2)} \rp^{s}\rp^{\m_1},
\end{equation}
where $\beta_{i_1}=\beta\|\x^{(i_1)}\|_2$ and
\begin{eqnarray}\label{eq:genanal7}
B^{(i_1,i_2)} & \triangleq &  \sqrt{t}(\y^{(i_2)})^T\u^{(i_1,1)}+\sqrt{1-t} (\y^{(i_2)})^T(b_1\u^{(2,1)}+b_2\u^{(2,2)}) \nonumber \\
D^{(i_1,i_2)} & \triangleq &  (B^{(i_1,i_2)}+\sqrt{t}\|\y^{(i_2)}\|_2 (a_1u^{(4,1)}+a_2u^{(4,2)})+\sqrt{1-t}\|\y^{(i_2)}\|_2(c_1\u^{(i_1,3,1)}+c_2\u^{(i_1,3,2)})) \nonumber \\
A^{(i_1,i_2)} & \triangleq &  e^{\beta_{i_1}D^{(i_1,i_2)}}\nonumber \\
C^{(i_1)} & \triangleq &  \sum_{i_2=1}^{l}A^{(i_1,i_2)}\nonumber \\
Z & \triangleq & \sum_{i_1=1}^{l} \lp \sum_{i_2=1}^{l} A^{(i_1,i_2)}\rp^s =\sum_{i_1=1}^{l}  (C^{(i_1)})^s.
\end{eqnarray}
Since our particular interest is the effect of $\p$ and $\q$ on the relation between $\psi(\calX,\calY,\p,\q,\m,\beta,s,1)$ and $\psi(\calX,\calY,\p,\q,\m,\beta,s,0)$, studying monotonicity of $\psi(\calX,\calY,\p,\q,\m,\beta,s,t)$ with respect to $\p$ and $\q$ seems as a good starting point. To that end, we start by considering its derivative
\begin{eqnarray}\label{eq:genanal9}
\frac{d\psi(\calX,\calY,\q,\m,\beta,s,t)}{d\p_1} & = &  \mE_{G,{\mathcal U}_2} \frac{1}{\beta|s|\sqrt{n}\m_1} \log \mE_{{\mathcal U}_1} Z^{\m_1}\nonumber \\
& = &  \mE_{G,{\mathcal U}_2} \frac{1}{\beta|s|\sqrt{n}\m_1\mE_{{\mathcal U}_1} Z^{\m_1}} \frac{d \mE_{{\mathcal U}_1} Z^{\m_1} }{d\p_1}\nonumber \\
& = &  \mE_{G,{\mathcal U}_2} \frac{\m_1}{\beta|s|\sqrt{n}\m_1\mE_{{\mathcal U}_1} Z^{\m_1}}\mE_{{\mathcal U}_1} \frac{1}{Z^{1-\m_1}}\frac{d Z}{d\p_1}\nonumber \\
& = &  \mE_{G,{\mathcal U}_2} \frac{\m_1}{\beta|s|\sqrt{n}\m_1\mE_{{\mathcal U}_1} Z^{\m_1}}\mE_{{\mathcal U}_1} \frac{1}{Z^{1-\m_1}} \frac{d\lp \sum_{i_1=1}^{l} \lp \sum_{i_2=1}^{l} A^{(i_1,i_2)}\rp^s \rp }{d\p_1}\nonumber \\
& = &   \mE_{G,{\mathcal U}_2} \frac{s\m_1}{\beta|s|\sqrt{n}\m_1\mE_{{\mathcal U}_1} Z^{\m_1}}\mE_{{\mathcal U}_1} \frac{1}{Z^{1-\m_1}}  \sum_{i=1}^{l} (C^{(i_1)})^{s-1} \nonumber \\
& & \times \sum_{i_2=1}^{l}\beta_{i_1}A^{(i_1,i_2)}\frac{dD^{(i_1,i_2)}}{d\p_1},
\end{eqnarray}
where
\begin{eqnarray}\label{eq:genanal9a}
\frac{dD^{(i_1,i_2)}}{d\p_1}= \lp \frac{dB^{(i_1,i_2)}}{d\p_1}-\sqrt{t}\q_1\frac{\|\y^{(i_2)}\|_2 u^{(4,1)}}{2\sqrt{\p_0\q_0-\p_1\q_1}}+\sqrt{t}\q_1\frac{\|\y^{(i_2)}\|_2 u^{(4,2)}}{2\sqrt{\p_1\q_1}} \rp.
\end{eqnarray}
Utilizing (\ref{eq:genanal7}) we also find
\begin{eqnarray}\label{eq:genanal10}
\frac{dB^{(i_1,i_2)}}{d\p_1} & = &   \frac{d\lp\sqrt{t}(\y^{(i_2)})^T\u^{(i_1,1)}+\sqrt{1-t} (\y^{(i_2)})^T(b_1\u^{(2,1)}+b_2\u^{(2,2)})\rp}{d\p_1} \nonumber \\
 & = &
\sum_{j=1}^{m}\lp -\sqrt{1-t}\frac{\y_j^{(i_2)}\u_j^{(2,1)}}{2\sqrt{\p_0-\p_1}}+\sqrt{1-t}\frac{\y_j^{(i_2)}\u_j^{(2,2)}}{2\sqrt{\p_1}}\rp.
\end{eqnarray}
Combining (\ref{eq:genanal9a}) and (\ref{eq:genanal10}) we obtain
\begin{eqnarray}\label{eq:genanal10a}
\frac{dD^{(i_1,i_2)}}{d\p_1} & = & \sum_{j=1}^{m}\lp -\sqrt{1-t}\frac{\y_j^{(i_2)}\u_j^{(2,1)}}{2\sqrt{\p_0-\p_1}}+\sqrt{1-t}\frac{\y_j^{(i_2)}\u_j^{(2,2)}}{2\sqrt{\p_1}}\rp \nonumber \\
& & -\sqrt{t}\q_1\frac{\|\y^{(i_2)}\|_2 u^{(4,1)}}{2\sqrt{\p_0\q_0-\p_1\q_1}}+\sqrt{t}\q_1\frac{\|\y^{(i_2)}\|_2 u^{(4,2)}}{2\sqrt{\p_1\q_1}} .
\end{eqnarray}
The above terms can be rearranged into two clearly distinguishable groups that depend on ${\mathcal U}_2$, and ${\mathcal U}_1$, respectively
\begin{eqnarray}\label{eq:genanal10b}
\frac{dD^{(i_1,i_2)}}{d\p_1} & = & \bar{T}_2+\bar{T}_1,
\end{eqnarray}
where
\begin{eqnarray}\label{eq:genanal10c}
 \bar{T}_2 & = & \sum_{j=1}^{m}\sqrt{1-t}\frac{\y_j^{(i_2)}\u_j^{(2,2)}}{2\sqrt{\p_1}}+\sqrt{t}\q_1\frac{\|\y^{(i_2)}\|_2 u^{(4,2)}}{2\sqrt{\p_1\q_1}} \nonumber\\
\bar{T}_1 & = &  -\sqrt{1-t}\sum_{j=1}^{m} \frac{\y_j^{(i_2)}\u_j^{(2,1)}}{2\sqrt{\p_0-\p_1}}  -\sqrt{t}\q_1\frac{\|\y^{(i_2)}\|_2 u^{(4,1)}}{2\sqrt{\p_0\q_0-\p_1\q_1}}.
\end{eqnarray}
Combining (\ref{eq:genanal9}) and (\ref{eq:genanal10b}) we have
\begin{equation}\label{eq:genanal10d}
\frac{d\psi(\calX,\calY,\q,\m,\beta,s,t)}{d\p_1}  =     \mE_{G,{\mathcal U}_2} \frac{s\m_1}{\beta|s|\sqrt{n}\m_1\mE_{{\mathcal U}_1} Z^{\m_1}}\mE_{{\mathcal U}_1} \frac{1}{Z^{1-\m_1}}  \sum_{i_1=1}^{l} (C^{(i_1)})^{s-1} \sum_{i_2=1}^{l}\beta_{i_1}A^{(i_1,i_2)}
\lp \bar{T}_2+ \bar{T}_1\rp.
\end{equation}
Finally, we recognize the following four objects placed in the two mentioned groups that will play a key role in the ensuing computations.
\begin{equation}\label{eq:genanal10e}
\frac{d\psi(\calX,\calY,\q,\m,\beta,s,t)}{d\p_1}  =       \frac{\mbox{sign}(s)}{2\beta\sqrt{n}} \sum_{i_1=1}^{l}  \sum_{i_2=1}^{l}
\beta_{i_1}\lp T_2^{\p}- T_1^{\p}\rp,
\end{equation}
where
\begin{eqnarray}\label{eq:genanal10f}
T_2^{\p} & = & \frac{\sqrt{1-t}}{\sqrt{\p_1}}\sum_{j=1}^{m} T_{2,1,j}^{\p} +\frac{\sqrt{t}\q_1}{\sqrt{\p_1\q_1}}\|\y^{(i_2)}\|_2T_{2,3}^{(\p,\q)} \nonumber\\
T_1^{\p} & = & \frac{\sqrt{1-t}}{\sqrt{\p_0-\p_1}}\sum_{j=1}^{m} T_{1,1,j}^{\p} +\frac{\sqrt{t}\q_1}{\sqrt{\p_0\q_0-\p_1\q_1}}\|\y^{(i_2)}\|_2T_{1,3}^{(\p,\q)}.
\end{eqnarray}
and
\begin{eqnarray}\label{eq:genanal10g}
 T_{2,1,j}^{\p} & = &  \mE_{G,{\mathcal U}_2}\lp \frac{1}{\mE_{{\mathcal U}_1} Z^{\m_1}}\mE_{{\mathcal U}_1}\frac{(C^{(i_1)})^{s-1} A^{(i_1,i_2)} \y_j^{(i_2)}\u_j^{(2,2)}}{Z^{1-\m_1}} \rp \nonumber \\
 T_{2,3}^{(\p,\q)} & = &  \mE_{G,{\mathcal U}_2}\lp \frac{1}{\mE_{{\mathcal U}_1} Z^{\m_1}}\mE_{{\mathcal U}_1}\frac{(C^{(i_1)})^{s-1} A^{(i_1,i_2)} u^{(4,2)}}{Z^{1-\m_1}} \rp \nonumber \\
T_{1,1,j}^{\p} & = &  \mE_{G,{\mathcal U}_2} \lp\frac{1}{\mE_{{\mathcal U}_1} Z^{\m_1}}\mE_{{\mathcal U}_1}\frac{(C^{(i_1)})^{s-1} A^{(i_1,i_2)} \y_j^{(i_2)}\u_j^{(2,1)}}{Z^{1-\m_1}}\rp \nonumber \\
 T_{1,3}^{(\p,\q)} & = &  \mE_{G,{\mathcal U}_2}\lp \frac{1}{\mE_{{\mathcal U}_1} Z^{\m_1}}\mE_{{\mathcal U}_1}\frac{(C^{(i_1)})^{s-1} A^{(i_1,i_2)} u^{(4,1)}}{Z^{1-\m_1}}\rp.
\end{eqnarray}
Analogously to (\ref{eq:genanal10e})-(\ref{eq:genanal10g}) we have
\begin{equation}\label{eq:genanal10e1}
\frac{d\psi(\calX,\calY,\q,\m,\beta,s,t)}{d\q_1}  =       \frac{\mbox{sign}(s)}{2\beta\sqrt{n}} \sum_{i_1=1}^{l}  \sum_{i_2=1}^{l}
\beta_{i_1}\lp T_2^{\q}- T_1^{\q}\rp,
\end{equation}
where
\begin{eqnarray}\label{eq:genanal10f1}
T_2^{\q} & = & \frac{\sqrt{1-t}}{\sqrt{\q_1}}\|\y^{(i_2)}\|_2 T_{2,2}^{\q} +\frac{\sqrt{t}\p_1}{\sqrt{\p_1\q_1}}\|\y^{(i_2)}\|_2T_{2,3}^{(\p,\q)} \nonumber\\
T_1^{\q} & = & \frac{\sqrt{1-t}}{\sqrt{\q_0-\q_1}} \|\y^{(i_2)}\|_2T_{1,2}^{\q} +\frac{\sqrt{t}\p_1}{\sqrt{\p_0\q_0-\p_1\q_1}}\|\y^{(i_2)}\|_2T_{1,3}^{(\p,\q)}.
\end{eqnarray}
and
\begin{eqnarray}\label{eq:genanal10g1}
 T_{2,2}^{\q} & = &  \mE_{G,{\mathcal U}_2}\lp \frac{1}{\mE_{{\mathcal U}_1} Z^{\m_1}}\mE_{{\mathcal U}_1}\frac{(C^{(i_1)})^{s-1} A^{(i_1,i_2)} \u^{(i_1,3,2)}}{Z^{1-\m_1}} \rp \nonumber \\
 T_{1,2}^{\q} & = &  \mE_{G,{\mathcal U}_2} \lp\frac{1}{\mE_{{\mathcal U}_1} Z^{\m_1}}\mE_{{\mathcal U}_1}\frac{(C^{(i_1)})^{s-1} A^{(i_1,i_2)} \u^{(i_1,3,1)}}{Z^{1-\m_1}}\rp.
 \end{eqnarray}
The above collection of equations (\ref{eq:genanal10d})-(\ref{eq:genanal10g1}) is, in principle, sufficient to determine the $\psi(\calX,\calY,\q,\m,\beta,s,t)$'s derivatives with respect to both $\p_1$ and $\q_1$. There are six key terms given in (\ref{eq:genanal10g}) and  (\ref{eq:genanal10g1}) and we will handle each of them separately. Throughout the process of handling each of them, we will heavily rely on  \cite{Stojnicgscompyx16,Stojnicnflgscompyx23} and, consequently, parallel their presentations as much as possible.

\subsection{Computing $\frac{d\psi(\calX,\calY,\p,\q,\m,\beta,s,t)}{d\p_1}$ and $\frac{d\psi(\calX,\calY,\p,\q,\m,\beta,s,t)}{d\q_1}$}
\label{sec:compderivative}

We carefully choose the order in which we handle the terms appearing in (\ref{eq:genanal10g}). In particular, we split the six terms into two groups of three. We first handle the three terms from the last group ($T_{1,1,j}^{\p}$, $T_{1,2}^{\q}$, and $T_{1,3}^{(\p,\q)}$) and  then the three terms from the first  group ($T_{2,1,j}^{\p}$,$ T_{2,2}^{\q}$, and $T_{2,3}^{(\p,\q)}$). Analogously to \cite{Stojnicnflgscompyx23}, we call the last group $T_1$--group and the first one $T_2$--group.

\subsubsection{Handling $T_1$--group}
\label{sec:handlT1}

We handle separately each of the three terms from $T_1$--group.

\underline{\textbf{\emph{Determining}} $T_{1,1,j}^{\p}$}
\label{sec:hand1T11}

We closely follow the presentation from Section  \ref{sec:hand1T11} of \cite{Stojnicnflgscompyx23} and after noting that the only difference is in the $b_1$ and $b_2$ scaling (and correspondingly rescaled $\u_j^{(2,1)}$ variances) we have through the Gaussian integration by parts
 \begin{eqnarray}\label{eq:liftgenAanal19}
T_{1,1,j}^{\p} & = & \mE_{G,{\mathcal U}_2}\lp \frac{1}{\mE_{{\mathcal U}_1} Z^{\m_1}}\mE_{{\mathcal U}_1}  \frac{(C^{(i_1)})^{s-1} A^{(i_1,i_2)}\y_j^{(i_2)}}{Z^{1-\m_1}}\rp \nonumber \\
& = & \mE_{G,{\mathcal U}_2} \lp \frac{1}{\mE_{{\mathcal U}_1} Z^{\m_1}}\mE_{{\mathcal U}_1}\lp \mE_{{\mathcal U}_1}(\u_j^{(2,1)}\u_j^{(2,1)}) \frac{d}{du_j^{(2,1)}}\lp\frac{(C^{(i_1)})^{s-1} A^{(i_1,i_2)}\y_j^{(i_2)}}{Z^{1-\m_1}}\rp\rp\rp \nonumber \\
& = & \mE_{G,{\mathcal U}_2}\lp \frac{1}{\mE_{{\mathcal U}_1} Z^{\m_1}}\mE_{{\mathcal U}_1}(\u_j^{(2,1)}\u_j^{(2,1)})\mE_{{\mathcal U}_1}\lp  \frac{d}{du_j^{(2,1)}}\lp\frac{(C^{(i_1)})^{s-1} A^{(i_1,i_2)}\y_j^{(i_2)}}{Z^{1-\m_1}}\rp\rp\rp \nonumber \\
& = & \mE_{G,{\mathcal U}_2} \lp\frac{1}{\mE_{{\mathcal U}_1} Z^{\m_1}}\mE_{{\mathcal U}_1}\lp  \frac{d}{du_j^{(2,1)}}\lp\frac{(C^{(i_1)})^{s-1} A^{(i_1,i_2)}\y_j^{(i_2)}}{Z^{1-\m_1}}\rp\rp\rp.
\end{eqnarray}
Since the last inner expectation is structurally identical to the one considered in Section \ref{sec:hand1T11} of \cite{Stojnicnflgscompyx23}, we can rely on the results obtained there. One needs to observe though that here $\u^{(2,1)}$ is  $\u^{(2,1)}$ scaled by $\sqrt{\p_0-\p_1}$. Following \cite{Stojnicnflgscompyx23} we then write
\begin{eqnarray}\label{eq:liftgenAanal19a}
T_{1,1,j}^{\p} & = &    \sqrt{\p_0-\p_1}\mE_{G,{\mathcal U}_2} \lp\frac{1}{\mE_{{\mathcal U}_1} Z^{\m_1}}\lp \Theta_1+\Theta_2 \rp\rp,
\end{eqnarray}
where  (as in \cite{Stojnicnflgscompyx23})
{\small\begin{eqnarray}\label{eq:liftgenAanal19c}
\Theta_1 &  = &  \mE_{{\mathcal U}_1} \Bigg( \Bigg. \frac{\y_j^{(i_2)} \lp (C^{(i_1)})^{s-1}\beta_{i_1}A^{(i_1,i_2)}\y_j^{(i_2)}\sqrt{1-t} +A^{(i_1,i_2)}(s-1)(C^{(i_1)})^{s-2}\beta_{i_1}\sum_{p_2=1}^{l}A^{(i_1,p_2)}\y_j^{(p_2)}\sqrt{1-t}\rp}{Z^{1-\m_1}}\Bigg. \Bigg)\Bigg. \Bigg) \nonumber \\
\Theta_2 & = & -(1-\m_1)\mE_{{\mathcal U}_1} \lp\sum_{p_1=1}^{l}
\frac{(C^{(i_1)})^{s-1} A^{(i_1,i_2)}\y_j^{(i_2)}}{Z^{2-\m_1}}
s  (C^{(p_1)})^{s-1}\sum_{p_2=1}^{l}\beta_{p_1}A^{(p_1,p_2)}\y_j^{(p_2)}\sqrt{1-t}\rp\Bigg.\Bigg).\nonumber \\
\end{eqnarray}}
and conveniently
\begin{eqnarray}\label{eq:liftgenAanal19d}
\sum_{i_1=1}^{l}\sum_{i_2=1}^{l}\sum_{j=1}^{m} \lp\frac{1}{\mE_{{\mathcal U}_1} Z^{\m_1}}\beta_{i_1}\Theta_1\rp
&  = & \sqrt{1-t} \lp  \mE_{{\mathcal U}_1}\frac{Z^{\m_1}}{\mE_{{\mathcal U}_1} Z^{\m_1}} \sum_{i_1=1}^{l}\frac{(C^{(i_1)})^s}{Z}\sum_{i_2=1}^{l}\frac{A^{(i_1,i_2)}}{C^{(i_1)}}\beta_{i_1}^2\|\y^{(i_2)}\|_2^2\rp \nonumber\\
& & +   \sqrt{1-t} \Bigg(\Bigg. \mE_{{\mathcal U}_1}\frac{Z^{\m_1}}{\mE_{{\mathcal U}_1} Z^{\m_1}} \nonumber \\
& & \times \sum_{i_1=1}^{l}\frac{(s-1)(C^{(i_1)})^s}{Z}\sum_{i_2=1}^{l}\sum_{p_2=1}^{l}\frac{A^{(i_1,i_2)}A^{(i_1,p_2)}}{(C^{(i_1)})^2}\beta_{i_1}^2(\y^{(p_2)})^T\y^{(i_2)}\Bigg.\Bigg).\nonumber \\
 \end{eqnarray}
After defining the operator
\begin{eqnarray}\label{eq:genAanal19e}
 \Phi_{{\mathcal U}_1} & \triangleq &  \mE_{{\mathcal U}_{1}} \frac{Z^{\m_1}}{\mE_{{\mathcal U}_{1}}Z^{\m_1}}  \triangleq  \mE_{{\mathcal U}_{1}} \lp \frac{Z^{\m_1}}{\mE_{{\mathcal U}_{1}}Z^{\m_1}}\lp \cdot \rp\rp,
 \end{eqnarray}
and setting measures
\begin{eqnarray}\label{eq:genAanal19e1}
  \gamma_0(i_1,i_2) & = &
\frac{(C^{(i_1)})^{s}}{Z}  \frac{A^{(i_1,i_2)}}{C^{(i_1)}} \nonumber \\
\gamma_{01}^{(1)}  & = &  \Phi_{{\mathcal U}_1} (\gamma_0(i_1,i_2)) \nonumber \\
\gamma_{02}^{(1)}  & = &  \Phi_{{\mathcal U}_1} (\gamma_0(i_1,i_2)\times \gamma_0(i_1,p_2)) \nonumber \\
\gamma_{1}^{(1)}   & = &  \Phi_{{\mathcal U}_1}  \lp \gamma_0(i_1,i_2)\times \gamma_0(p_1,p_2) \rp \nonumber \\.
\gamma_{2}^{(1)}   & = &  \Phi_{{\mathcal U}_1}  \gamma_0(i_1,i_2)  \times   \Phi_{{\mathcal U}_1} \gamma_0(p_1,p_2).
\end{eqnarray}
 we, from (\ref{eq:liftgenAanal19d}), find
\begin{eqnarray}\label{eq:liftgenAanal19g}
\sum_{i_1=1}^{l}\sum_{i_2=1}^{l}\sum_{j=1}^{m} \lp\frac{1}{\mE_{{\mathcal U}_1} Z^{\m_1}} \beta_{i_1}\Theta_1\rp
&  = &  \sqrt{1-t} \beta^2 \lp \langle \|\x^{(i_1)}\|_2^2\|\y^{(i_2)}\|_2^2\rangle_{\gamma_{01}^{(1)}} +  (s-1) \langle \|\x^{(i_1)}\|_2^2(\y^{(p_2)})^T\y^{(i_2)}\rangle_{\gamma_{02}^{(1)}} \rp. \nonumber \\
 \end{eqnarray}
From (\ref{eq:liftgenAanal19c}), we also have
\begin{eqnarray}\label{eq:liftgenAanal19h}
\sum_{i_1=1}^{l}\sum_{i_2=1}^{l}\sum_{j=1}^{m} \lp\frac{1}{\mE_{{\mathcal U}_1} Z^{\m_1}} \beta_{i_1}\Theta_2\rp
 & = & - \sqrt{1-t} s(1-\m_1) \mE_{G,{\mathcal U}_2} \Bigg( \Bigg. \frac{Z^{\m_1}}{\mE_{{\mathcal U}_1} Z^{\m_1}} \sum_{i_1=1}^{l}\frac{(C^{(i_1)})^s}{Z}\sum_{i_2=1}^{l}
\frac{A^{(i_1,i_2)}}{C^{(i_1)}} \nonumber \\
& & \times
 \sum_{p_1=1}^{l} \frac{(C^{(p_1)})^s}{Z}\sum_{p_2=1}^{l}\frac{A^{(p_1,p_2)}}{C^{(p_1)}} \beta_{i_1}\beta_{p_1}(\y^{(p_2)})^T\y^{(i_2)} \Bigg.\Bigg)\nonumber \\
& =& - \sqrt{1-t} s\beta^2(1-\m_1) \mE_{G,{\mathcal U}_2} \langle \|\x^{(i_1)}\|_2\|\x^{(p_1)}\|_2(\y^{(p_2)})^T\y^{(i_2)} \rangle_{\gamma_{1}^{(1)}},
\end{eqnarray}
and after combining  (\ref{eq:liftgenAanal19a}), (\ref{eq:liftgenAanal19g}), and (\ref{eq:liftgenAanal19h}) we have
\begin{eqnarray}\label{eq:liftgenAanal19i}
\sum_{i_1=1}^{l}\sum_{i_2=1}^{l}\sum_{j=1}^{m} \beta_{i_1}\frac{\sqrt{1-t}}{\sqrt{\p_0-\p_1}}T_{1,1,j}^{\p}
& = &  \sqrt{1-t} \mE_{G,{\mathcal U}_2} \lp\frac{1}{\mE_{{\mathcal U}_1} Z^{\m_1}}\lp \beta_{i_1}\Theta_1+\beta_{i_1}\Theta_2 \rp\rp\nonumber \\
& = & (1-t)\beta^2 \nonumber \\
 & &
 \times \lp \mE_{G,{\mathcal U}_2} \langle \|\x^{(i_1)}\|_2^2\|\y^{(i_2)}\|_2^2\rangle_{\gamma_{01}^{(1)}} +  (s-1)\mE_{G,{\mathcal U}_2}\langle \|\x^{(i_1)}\|_2^2(\y^{(p_2)})^T\y^{(i_2)}\rangle_{\gamma_{02}^{(1)}} \rp \nonumber \\
& & - (1-t)s\beta^2(1-\m_1)  \mE_{G,{\mathcal U}_2}  \langle \|\x^{(i_1)}\|_2\|\x^{(p_1)}\|_2(\y^{(p_2)})^T\y^{(i_2)} \rangle_{\gamma_{1}^{(1)}}.
\end{eqnarray}

As a side note, we observe that the  $\gamma$ measures from (\ref{eq:genAanal19e1}) (as well as the $\Phi(\cdot)$ operators from (\ref{eq:genAanal19e})) are  functions of $t$. One then trivially has that all functions that depend on $\gamma$'s, that will be utilized throughout the paper, are functions of $t$ as well. To keep notation as light as it can be, we skip stating this $t$-dependence explicitly, and, instead, assume it implicitly.

\underline{\textbf{\emph{Determining}} $T_{1,2}^{\q}$}
\label{sec:hand1T12}

As above, we follow the Gaussian integration by parts path to find
\begin{eqnarray}\label{eq:liftgenBanal20}
T_{1,2}^{\q} & = & \mE_{G,{\mathcal U}_2} \lp \frac{1}{\mE_{{\mathcal U}_1} Z^{\m_1}}\mE_{{\mathcal U}_1} \frac{(C^{(i_1)})^{s-1} A^{(i_1,i_2)}\u^{(i_1,3,1)}}{Z^{1-\m_1}}\rp \nonumber \\
& = & \mE_{G,{\mathcal U}_2} \lp \frac{1}{\mE_{{\mathcal U}_1} Z^{\m_1}}
\mE_{{\mathcal U}_1} \sum_{p_1=1}^{l}\mE_{{\mathcal U}_1}(\u^{(i_1,3,1)}\u^{(p_1,3,1)}) \frac{d}{d\u^{(p_1,3,1)}}\lp\frac{(C^{(i_1)})^{s-1} A^{(i_1,i_2)}}{Z^{1-\m_1}}\rp\rp \nonumber \\
& = & \frac{\sqrt{\q_0-\q_1}}{\q_0-\q_1}\mE_{G,{\mathcal U}_2} \Bigg(\Bigg.  \frac{1}{\mE_{{\mathcal U}_1} Z^{\m_1}}
\mE_{{\mathcal U}_1} \sum_{p_1=1}^{l}\frac{(\sqrt{\q_0-\q_1}\x^{(i_1)})^T\sqrt{\q_0-\q_1}\x^{(p_1)}}{\|\x^{(i_1)}\|_2\|\x^{(p_1)}\|_2} \nonumber \\
& & \times
\frac{d}{d\lp \sqrt{\q_0-\q_1} \u^{(p_1,3,1)}\rp}\lp\frac{(C^{(i_1)})^{s-1} A^{(i_1,i_2)}}{Z^{1-\m_1}}\rp\Bigg.\Bigg) \nonumber\\
& = & \frac{\sqrt{\q_0-\q_1}}{\q_0-\q_1} T_{1,2}
\end{eqnarray}
where   \cite{Stojnicnflgscompyx23} determined
\begin{eqnarray}\label{eq:liftgenBanal20a0}
\sum_{i_1=1}^{l}\sum_{i_2=1}^{l} \beta_{i_1}\|\y^{(i_2)}\|_2 T_{1,2} & = & \sqrt{1-t}(\q_0-\q_1)\beta^2
\Bigg( \Bigg.\mE_{G,{\mathcal U}_2}\langle \|\x^{(i_1)}\|_2^2\|\y^{(i_2)}\|_2^2\rangle_{\gamma_{01}^{(1)}} \nonumber \\
& & +   (s-1)\mE_{G,{\mathcal U}_2}\langle \|\x^{(i_1)}\|_2^2 \|\y^{(i_2)}\|_2\|\y^{(p_2)}\|_2\rangle_{\gamma_{02}^{(1)}}\Bigg.\Bigg)  \nonumber \\
& & - \sqrt{1-t}(\q_0-\q_1)s\beta^2(1-\m_1)\mE_{G,{\mathcal U}_2}\langle (\x^{(p_1)})^T\x^{(i_1)}\|\y^{(i_2)}\|_2\|\y^{(p_2)}\|_2 \rangle_{\gamma_{1}^{(1)}}.\nonumber \\
\end{eqnarray}
Combining (\ref{eq:liftgenBanal20}) and (\ref{eq:liftgenBanal20a0}), we can then immediately write
\begin{eqnarray}\label{eq:liftgenBanal20b}
\sum_{i_1=1}^{l}\sum_{i_2=1}^{l} \beta_{i_1}\|\y^{(i_2)}\|_2 \frac{\sqrt{1-t}}{\sqrt{\q_0-\q_1}}T_{1,2}^{\q} & = & (1-t)\beta^2
\Bigg( \Bigg.\mE_{G,{\mathcal U}_2}\langle \|\x^{(i_1)}\|_2^2\|\y^{(i_2)}\|_2^2\rangle_{\gamma_{01}^{(1)}} \nonumber \\
& & +   (s-1)\mE_{G,{\mathcal U}_2}\langle \|\x^{(i_1)}\|_2^2 \|\y^{(i_2)}\|_2\|\y^{(p_2)}\|_2\rangle_{\gamma_{02}^{(1)}}\Bigg.\Bigg)  \nonumber \\
& & - (1-t)s\beta^2(1-\m_1)\mE_{G,{\mathcal U}_2}\langle (\x^{(p_1)})^T\x^{(i_1)}\|\y^{(i_2)}\|_2\|\y^{(p_2)}\|_2 \rangle_{\gamma_{1}^{(1)}}.\nonumber \\
\end{eqnarray}

\underline{\textbf{\emph{Determining}} $T_{1,3}^{(\p,\q)}$}
\label{sec:hand1T13}

Proceeding as in \cite{Stojnicnflgscompyx23} and relying on the Gaussian integration by parts, we obtain
\begin{eqnarray}\label{eq:liftgenCanal21}
T_{1,3}^{(\p,\q)} & = & \mE_{G,{\mathcal U}_2} \lp \frac{1}{\mE_{{\mathcal U}_1} Z^{\m_1}}\mE_{{\mathcal U}_1}  \frac{(C^{(i_1)})^{s-1} A^{(i_1,i_2)}u^{(4,1)}}{Z^{1-\m_1}} \rp \nonumber \\
& = & \mE_{G,{\mathcal U}_2} \lp \frac{1}{\mE_{{\mathcal U}_1} Z^{\m_1}}\mE_{{\mathcal U}_1} \lp\mE_{{\mathcal U}_1} (u^{(4,1)}u^{(4,1)})\lp\frac{d}{du^{(4,1)}} \lp\frac{(C^{(i_1)})^{s-1} A^{(i_1,i_2)}u^{(4,1)}}{Z^{1-\m_1}}\rp \rp\rp\rp \nonumber \\
 & = & \frac{\sqrt{\p_0\q_0-\p_1\q_1}}{\p_0\q_0-\p_1\q_1}\mE_{G,{\mathcal U}_2} \Bigg( \Bigg.\frac{\mE_{{\mathcal U}_1} (\sqrt{\p_0\q_0-\p_1\q_1}u^{(4,1)}\sqrt{\p_0\q_0-\p_1\q_1}u^{(4,1)})}{\mE_{{\mathcal U}_1} Z^{\m_1}} \nonumber \\
 & & \times \mE_{{\mathcal U}_1} \lp\frac{d}{d \lp \sqrt{\p_0\q_0-\p_1\q_1} u^{(4,1)} \rp} \lp\frac{(C^{(i_1)})^{s-1} A^{(i_1,i_2)}u^{(4,1)}}{Z^{1-\m_1}}\rp\rp\Bigg.\Bigg) \nonumber \\
 & = & \frac{1}{\sqrt{\p_0\q_0-\p_1\q_1}}T_{1,3},
\end{eqnarray}
where   \cite{Stojnicnflgscompyx23} obtained
\begin{eqnarray}\label{eq:liftgenCanal21b01}
\sum_{i_1=1}^{l}\sum_{i_2=1}^{l} \beta_{i_1}\|\y^{(i_2)}\|_2 T_{1,2} & = & \sqrt{t}(\p_0\q_0-\p_1\q_1)\beta^2 \Bigg( \Bigg. \mE_{G,{\mathcal U}_2}\langle \|\x^{(i_1)}\|_2^2\|\y^{(i_2)}\|_2^2\rangle_{\gamma_{01}^{(1)}} \nonumber \\
& & +   (s-1)\mE_{G,{\mathcal U}_2}\langle \|\x^{(i_1)}\|_2^2 \|\y^{(i_2)}\|_2\|\y^{(p_2)}\|_2\rangle_{\gamma_{02}^{(1)}}\Bigg.\Bigg) \nonumber \\
& & - \sqrt{t}(\p_0\q_0-\p_1\q_1)s\beta^2(1-\m_1)\mE_{G,{\mathcal U}_2}\langle \|\x^{(i_1)}\|_2\|\x^{(p_`)}\|_2\|\y^{(i_2)}\|_2\|\y^{(p_2)}\|_2 \rangle_{\gamma_{1}^{(1)}}. \nonumber \\
\end{eqnarray}
Analogously to (\ref{eq:liftgenBanal20b}), we have
\begin{eqnarray}\label{eq:liftgenCanal21b}
\sum_{i_1=1}^{l}\sum_{i_2=1}^{l} \beta_{i_1}\|\y^{(i_2)}\|_2 \frac{\sqrt{t}}{\sqrt{\p_0\q_0-\p_1\q_1}}T_{1,3}^{(\p,\q)} & = & t\beta^2 \Bigg( \Bigg. \mE_{G,{\mathcal U}_2}\langle \|\x^{(i_1)}\|_2^2\|\y^{(i_2)}\|_2^2\rangle_{\gamma_{01}^{(1)}} \nonumber \\
& & +   (s-1)\mE_{G,{\mathcal U}_2}\langle \|\x^{(i_1)}\|_2^2 \|\y^{(i_2)}\|_2\|\y^{(p_2)}\|_2\rangle_{\gamma_{02}^{(1)}}\Bigg.\Bigg) \nonumber \\
& & - ts\beta^2(1-\m_1)\mE_{G,{\mathcal U}_2}\langle \|\x^{(i_1)}\|_2\|\x^{(p_`)}\|_2\|\y^{(i_2)}\|_2\|\y^{(p_2)}\|_2 \rangle_{\gamma_{1}^{(1)}}. \nonumber \\
\end{eqnarray}

\subsubsection{Handling $T_2$--group}
\label{sec:handlT2}

Similarly to the above handling of $T_1$--group, we handle separately each of the three terms that $T_2$'s contribution is comprised of.

\underline{\textbf{\emph{Determining}} $T_{2,1,j}^{\p}$}
\label{sec:hand1T21}

Utilizing the Gaussian integration by parts, we have
\begin{eqnarray}\label{eq:genDanal19}
T_{2,1,j}^{\p}& = &  \mE_{G,{\mathcal U}_2}\lp \frac{1}{\mE_{{\mathcal U}_1} Z^{\m_1}}\mE_{{\mathcal U}_1}\frac{(C^{(i_1)})^{s-1} A^{(i_1,i_2)} \y_j^{(i_2)}\u_j^{(2,2)}}{Z^{1-\m_1}} \rp \nonumber \\
 &  = & \mE_{G,{\mathcal U}_2,{\mathcal U}_1}  \frac{(C^{(i_1)})^{s-1} A^{(i_1,i_2)}\y_j^{(i_2)}\u_j^{(2,2)}}{Z^{1-\m_1}\mE_{{\mathcal U}_1} Z^{\m_1}} \nonumber \\
& = &
\mE_{G,{\mathcal U}_1}\lp\mE_{{\mathcal U}_2}\lp\mE_{{\mathcal U}_2} (\u_j^{(2,2)}\u_j^{(2,2)})\frac{d}{d\u_j^{(2,2)}}\lp \frac{(C^{(i_1)})^{s-1} A^{(i_1,i_2)}\y_j^{(i_2)}}{Z^{1-\m_1}\mE_{{\mathcal U}_1} Z^{\m_1}}\rp\rp\rp \nonumber \\
& = &
\frac{\sqrt{\p_1}}{\p_1}
\Bigg( \Bigg.
\mE_{G,{\mathcal U}_2,{\mathcal U}_1}\lp \frac{\mE_{{\mathcal U}_2} (\sqrt{\p_1}\u_j^{(2,2)}\sqrt{\p_1}\u_j^{(2,2)})}{\mE_{{\mathcal U}_1} Z^{\m_1}}\frac{d}{d\lp \sqrt{\p_1}\u_j^{(2,2)}\rp}\lp \frac{(C^{(i_1)})^{s-1} A^{(i_1,i_2)}\y_j^{(i_2)}}{Z^{1-\m_1}}\rp\rp \nonumber \\
& & + \mE_{G,{\mathcal U}_2,{\mathcal U}_1}\lp \frac{\mE_{{\mathcal U}_2} (\sqrt{\p_1}\u_j^{(2,2)}\sqrt{\p_1}\u_j^{(2,2)})\lp(C^{(i_1)})^{s-1} A^{(i_1,i_2)}\y_j^{(i_2)} \rp}{Z^{1-\m_1}}\frac{d}{d\lp\sqrt{\p_1} \u_j^{(2,2)}\rp}\lp \frac{1}{\mE_{{\mathcal U}_1} Z^{\m_1}}\rp\rp  \Bigg. \Bigg) \nonumber \\
& = & \frac{\sqrt{\p_1}}{\p_1}T_{2,1,j},
\end{eqnarray}
where \cite{Stojnicnflgscompyx23} determined
 \begin{eqnarray}\label{eq:genDanal25a01}
 \sum_{i_1=1}^{l}  \sum_{i_2=1}^{l} \sum_{j=1}^{m}  \beta_{i_1}T_{2,1,j}& = & \sqrt{1-t}\p_1\beta^2
 \Bigg(\Bigg. \mE_{G,{\mathcal U}_2}\langle \|\x^{(i_1)}\|_2^2\|\y^{(i_2)}\|_2^2\rangle_{\gamma_{01}^{(1)}} \nonumber \\
 & & +  (s-1)\mE_{G,{\mathcal U}_2}\langle \|\x^{(i_1)}\|_2^2(\y^{(p_2)})^T\y^{(i_2)}\rangle_{\gamma_{02}^{(1)}} \Bigg.\Bigg) \nonumber \\
& & - \sqrt{1-t}\p_1s\beta^2(1-\m_1)\mE_{G,{\mathcal U}_2}\langle \|\x^{(i_1)}\|_2\|\x^{(p_1)}\|_2(\y^{(p_2)})^T\y^{(i_2)} \rangle_{\gamma_{1}^{(1)}}\nonumber \\
 &   &
  -\sqrt{1-t}\p_1s\beta^2\m_1\mE_{G,{\mathcal U}_2} \langle \|\x^{(i_1)}\|_2\|\x^{(p_1)}\|_2(\y^{(p_2)})^T\y^{(i_2)} \rangle_{\gamma_{2}^{(1)}}.
\end{eqnarray}
A combination of (\ref{eq:genDanal19}) and (\ref{eq:genDanal25a01}) gives
\begin{eqnarray}\label{eq:genDanal25}
 \sum_{i_1=1}^{l}  \sum_{i_2=1}^{l} \sum_{j=1}^{m}  \beta_{i_1}\frac{\sqrt{1-t}}{\sqrt{\p_1}}T_{2,1,j}^{\p}& = & (1-t)\beta^2
 \Bigg(\Bigg. \mE_{G,{\mathcal U}_2}\langle \|\x^{(i_1)}\|_2^2\|\y^{(i_2)}\|_2^2\rangle_{\gamma_{01}^{(1)}} \nonumber \\
 & & +  (s-1)\mE_{G,{\mathcal U}_2}\langle \|\x^{(i_1)}\|_2^2(\y^{(p_2)})^T\y^{(i_2)}\rangle_{\gamma_{02}^{(1)}} \Bigg.\Bigg) \nonumber \\
& & - (1-t)s\beta^2(1-\m_1)\mE_{G,{\mathcal U}_2}\langle \|\x^{(i_1)}\|_2\|\x^{(p_1)}\|_2(\y^{(p_2)})^T\y^{(i_2)} \rangle_{\gamma_{1}^{(1)}}\nonumber \\
 &   &
  -(1-t)s\beta^2\m_1\mE_{G,{\mathcal U}_2} \langle \|\x^{(i_1)}\|_2\|\x^{(p_1)}\|_2(\y^{(p_2)})^T\y^{(i_2)} \rangle_{\gamma_{2}^{(1)}}.
\end{eqnarray}

\underline{\textbf{\emph{Determining}} $T_{2,2}^{\q}$}
\label{sec:hand1T22}

Applying again Gaussian integration by parts, we obtain
{\small \begin{eqnarray}\label{eq:liftgenEanal20}
T_{2,2}^{\q} & = &  \mE_{G,{\mathcal U}_2} \lp \frac{1}{\mE_{{\mathcal U}_1} Z^{\m_1}}\mE_{{\mathcal U}_1}\frac{(C^{(i_1)})^{s-1} A^{(i_1,i_2)} \u^{(i_1,3,2)}}{Z^{1-\m_1}} \rp \nonumber \\
& = &  \mE_{G,{\mathcal U}_2,{\mathcal U}_1} \lp\frac{1}{\mE_{{\mathcal U}_1} Z^{\m_1}} \frac{(C^{(i_1)})^{s-1} A^{(i_1,i_2)} \u^{(i_1,3,2)}}{Z^{1-\m_1}}\rp\nonumber \\
& = & \mE_{G,{\mathcal U}_1} \lp
\mE_{{\mathcal U}_2} \lp \sum_{p_1=1}^{l}\mE_{{\mathcal U}_2}(\u^{(i_1,3,2)}\u^{(p_1,3,2)}) \frac{d}{d\u^{(p_1,3,2)}}\lp\frac{(C^{(i_1)})^{s-1} A^{(i_1,i_2)}}{Z^{1-\m_1}\mE_{{\mathcal U}_1} Z^{\m_1}}\rp\rp\rp \nonumber \\
& = & \frac{\sqrt{\q_1}}{\q_1}
\Bigg( \Bigg.
\mE_{G,{\mathcal U}_2,{\mathcal U}_1} \lp
 \frac{1}{\mE_{{\mathcal U}_1} Z^{\m_1}} \sum_{p_1=1}^{l}\mE_{{\mathcal U}_2}(\sqrt{\q_1}\u^{(i_1,3,2)}\sqrt{\q_1}\u^{(p_1,3,2)}) \frac{d}{d \lp \sqrt{\q_1}\u^{(p_1,3,2)}\rp}\lp\frac{(C^{(i_1)})^{s-1} A^{(i_1,i_2)}}{Z^{1-\m_1}}\rp\rp \nonumber \\
& & + \mE_{G,{\mathcal U}_2,{\mathcal U}_1} \lp \frac{(C^{(i_1)})^{s-1} A^{(i_1,i_2)}}{Z^{1-\m_1}}
  \sum_{p_1=1}^{l}\mE_{{\mathcal U}_2}(\sqrt{\q_1}\u^{(i_1,3,2)}\sqrt{\q_1}\u^{(p_1,3,2)}) \frac{d}{d \lp\sqrt{\q_1}  \u^{(p_1,3,2)}\rp}\lp\frac{1}{\mE_{{\mathcal U}_1} Z^{\m_1}}\rp\rp  \Bigg. \Bigg) \nonumber \\
  & = & \frac{\sqrt{\q_1}}{\q_1} T_{2,2},
 \end{eqnarray}}

\noindent where \cite{Stojnicnflgscompyx23} obtained
 \begin{eqnarray}\label{eq:genEanal25a01}
\sum_{i_1=1}^{l}\sum_{i_2=1}^{l} \beta_{i_1}\|\y^{(i_2)}\|_2 T_{2,2} &  = &
\sqrt{1-t}\q_1\beta^2 \Bigg( \Bigg. \mE_{G,{\mathcal U}_2}\langle \|\x^{(i_1)}\|_2^2\|\y^{(i_2)}\|_2^2\rangle_{\gamma_{01}^{(1)}} \nonumber \\
& & +  (s-1)\mE_{G,{\mathcal U}_2}\langle \|\x^{(i_1)}\|_2^2 \|\y^{(i_2)}\|_2\|\y^{(p_2)}\|_2\rangle_{\gamma_{02}^{(1)}}\Bigg.\Bigg) \nonumber \\
& & - \sqrt{1-t}\q_1s\beta^2(1-\m_1)\mE_{G,{\mathcal U}_2}\langle (\x^{(p_1)})^T\x^{(i_1)}\|\y^{(i_2)}\|_2\|\y^{(p_2)}\|_2 \rangle_{\gamma_{1}^{(1)}} \nonumber \\
&  & -\sqrt{1-t}\q_1s\beta^2\m_1\mE_{G,{\mathcal U}_2} \langle \|\y^{(i_2)}\|_2\|\y^{(p_2)}\|_2(\x^{(i_1)})^T\x^{(p_1)}\rangle_{\gamma_{2}^{(1)}}.
\end{eqnarray}
Combining (\ref{eq:liftgenEanal20}) and (\ref{eq:genEanal25a01}) we find
  \begin{eqnarray}\label{eq:genEanal25}
\sum_{i_1=1}^{l}\sum_{i_2=1}^{l} \beta_{i_1}\|\y^{(i_2)}\|_2\frac{\sqrt{1-t}}{\sqrt{\q_1}}T_{2,2}^{\q} &  = &
(1-t)\beta^2 \Bigg( \Bigg. \mE_{G,{\mathcal U}_2}\langle \|\x^{(i_1)}\|_2^2\|\y^{(i_2)}\|_2^2\rangle_{\gamma_{01}^{(1)}} \nonumber \\
& & +  (s-1)\mE_{G,{\mathcal U}_2}\langle \|\x^{(i_1)}\|_2^2 \|\y^{(i_2)}\|_2\|\y^{(p_2)}\|_2\rangle_{\gamma_{02}^{(1)}}\Bigg.\Bigg) \nonumber \\
& & - \q_1s\beta^2(1-\m_1)\mE_{G,{\mathcal U}_2}\langle (\x^{(p_1)})^T\x^{(i_1)}\|\y^{(i_2)}\|_2\|\y^{(p_2)}\|_2 \rangle_{\gamma_{1}^{(1)}} \nonumber \\
&  & -(1-t)s\beta^2\m_1\mE_{G,{\mathcal U}_2} \langle \|\y^{(i_2)}\|_2\|\y^{(p_2)}\|_2(\x^{(i_1)})^T\x^{(p_1)}\rangle_{\gamma_{2}^{(1)}}.
\end{eqnarray}

\underline{\textbf{\emph{Determining}} $T_{2,3}^{(\p,\q)}$}
\label{sec:hand1T23}

Gaussian integration by parts also gives
\begin{eqnarray}\label{eq:genFanal21}
T_{2,3}^{(\p,\q)} & = &  \mE_{G,{\mathcal U}_2}\lp \frac{1}{\mE_{{\mathcal U}_1} Z^{\m_1}}\mE_{{\mathcal U}_1}\frac{(C^{(i_1)})^{s-1} A^{(i_1,i_2)} u^{(4,2)}}{Z^{1-\m_1}} \rp \nonumber \\
& = & \mE_{G,{\mathcal U}_1} \lp  \mE_{{\mathcal U}_2} \lp\mE_{{\mathcal U}_2} (u^{(4,2)}u^{(4,2)})\lp\frac{d}{du^{(4,2)}} \lp\frac{(C^{(i_1)})^{s-1} A^{(i_1,i_2)}}{Z^{1-\m_1}\mE_{{\mathcal U}_1} Z^{\m_1}}\rp \rp\rp\rp \nonumber \\
& = & \frac{\sqrt{\p_1\q_1}}{\p_1\q_1}\mE_{G,{\mathcal U}_2,{\mathcal U}_1} \lp \frac{\mE_{{\mathcal U}_2} (\sqrt{\p_1\q_1}u^{(4,2)}\sqrt{\p_1\q_1}u^{(4,2)})}{\mE_{{\mathcal U}_1} Z^{\m_1}} \lp\frac{d}{d \lp \sqrt{\p_1\q_1}u^{(4,2)}\rp} \lp\frac{(C^{(i_1)})^{s-1} A^{(i_1,i_2)}}{Z^{1-\m_1}}\rp\rp\rp \nonumber \\
& & + \frac{\sqrt{\p_1\q_1}}{\p_1\q_1} \nonumber \\
& & \times \mE_{G,{\mathcal U}_2,{\mathcal U}_1} \lp \frac{\mE_{{\mathcal U}_2} (\sqrt{\p_1\q_1}u^{(4,2)}\sqrt{\p_1\q_1}u^{(4,2)})(C^{(i_1)})^{s-1} A^{(i_1,i_2)}}{Z^{1-\m_1}}\lp\frac{d}{d\lp \sqrt{\p_1\q_1} u^{(4,2)}\rp} \lp\frac{1}{\mE_{{\mathcal U}_1} Z^{\m_1}} \rp\rp\rp \nonumber \\
& = & \frac{\sqrt{\p_1\q_1}}{\p_1\q_1} T_{2,3},
\end{eqnarray}
where \cite{Stojnicnflgscompyx23} obtained
 \begin{eqnarray}\label{eq:genFanal29a01}
\sum_{i_1=1}^{l}\sum_{i_2=1}^{l} \beta_{i_1}\|\y^{(i_2)}\|_2 T_{2,3}
& = &
\sqrt{t} \p_1\q_1\beta^2 \Bigg( \Bigg. \mE_{G,{\mathcal U}_2}\langle \|\x^{(i_1)}\|_2^2\|\y^{(i_2)}\|_2^2\rangle_{\gamma_{01}^{(1)}} \nonumber \\
& & +   (s-1)\mE_{G,{\mathcal U}_2}\langle \|\x^{(i_1)}\|_2^2 \|\y^{(i_2)}\|_2\|\y^{(p_2)}\|_2\rangle_{\gamma_{02}^{(1)}}\Bigg.\Bigg) \nonumber \\
& & - \sqrt{t}\p_1\q_1 s\beta^2(1-\m_1)\mE_{G,{\mathcal U}_2}\langle \|\x^{(i_1)}\|_2\|\x^{(p_`)}\|_2\|\y^{(i_2)}\|_2\|\y^{(p_2)}\|_2 \rangle_{\gamma_{1}^{(1)}} \nonumber \\
&  & -\sqrt{t}s\beta^2\p_1\q_1\m_1\mE_{G,{\mathcal U}_2} \langle\|\x^{(i_2)}\|_2\|\x^{(p_2)}\|_2\|\y^{(i_2)}\|_2\|\y^{(p_2)}\rangle_{\gamma_{2}^{(1)}}.
\end{eqnarray}
Combining (\ref{eq:genFanal21}) and (\ref{eq:genFanal29a01}), we find
 \begin{eqnarray}\label{eq:genFanal29}
\sum_{i_1=1}^{l}\sum_{i_2=1}^{l} \beta_{i_1}\|\y^{(i_2)}\|_2\frac{\sqrt{t}}{\sqrt{\p_1\q_1}}T_{2,3}^{(\p,\q)}
& = &
t\beta^2 \Bigg( \Bigg. \mE_{G,{\mathcal U}_2}\langle \|\x^{(i_1)}\|_2^2\|\y^{(i_2)}\|_2^2\rangle_{\gamma_{01}^{(1)}} \nonumber \\
& & +   (s-1)\mE_{G,{\mathcal U}_2}\langle \|\x^{(i_1)}\|_2^2 \|\y^{(i_2)}\|_2\|\y^{(p_2)}\|_2\rangle_{\gamma_{02}^{(1)}}\Bigg.\Bigg) \nonumber \\
& & - t s\beta^2(1-\m_1)\mE_{G,{\mathcal U}_2}\langle \|\x^{(i_1)}\|_2\|\x^{(p_`)}\|_2\|\y^{(i_2)}\|_2\|\y^{(p_2)}\|_2 \rangle_{\gamma_{1}^{(1)}} \nonumber \\
&  & -ts\beta^2\m_1\mE_{G,{\mathcal U}_2} \langle\|\x^{(i_2)}\|_2\|\x^{(p_2)}\|_2\|\y^{(i_2)}\|_2\|\y^{(p_2)}\rangle_{\gamma_{2}^{(1)}}.
\end{eqnarray}

\subsubsection{Connecting all pieces together}
\label{sec:conalt}

We now connect together all the pieces obtained above. To do so, we utilize (\ref{eq:genanal10e}) and (\ref{eq:genanal10f}) to write
\begin{eqnarray}\label{eq:ctp1}
\frac{d\psi(\calX,\calY,\q,\m,\beta,s,t)}{d\p_1}  & = &       \frac{\mbox{sign}(s)}{2\beta\sqrt{n}} \lp \Omega_1+\q_1\Omega_3\rp \nonumber \\
\frac{d\psi(\calX,\calY,\q,\m,\beta,s,t)}{d\q_1}  & = &       \frac{\mbox{sign}(s)}{2\beta\sqrt{n}} \lp \Omega_2+\p_1\Omega_3\rp,
\end{eqnarray}
where
\begin{eqnarray}\label{eq:ctp2}
\Omega_1 & = & \sum_{i_1=1}^{l}  \sum_{i_2=1}^{l} \sum_{j=1}^{m}\beta_{i_1}\frac{\sqrt{1-t}}{\sqrt{\p_1}}T_{2,1,j}^{\p}-\sum_{i_1=1}^{l}  \sum_{i_2=1}^{l} \sum_{j=1}^{m}\beta_{i_1}\frac{\sqrt{1-t}}{\sqrt{\p_0-\p_1}}T_{1,1,j}^{\p} \nonumber\\
\Omega_2 & = & \sum_{i_1=1}^{l}  \sum_{i_2=1}^{l}\beta_{i_1}\|\y^{(i_2)}\|_2\frac{\sqrt{1-t}}{\sqrt{\q_1}}T_{2,2}^{\q}-\sum_{i_1=1}^{l}  \sum_{i_2=1}^{l}\beta_{i_1}\|\y^{(i_2)}\|_2\frac{\sqrt{1-t}}{\q_0-\q_1}T_{1,2}^{\q} \nonumber\\
\Omega_3 & = & \sum_{i_1=1}^{l}  \sum_{i_2=1}^{l}\beta_{i_1}\|\y^{(i_2)}\|_2\frac{\sqrt{t}}{\sqrt{\p_1\q_1}}T_{2,3}^{(\p,\q)}- \sum_{i_1=1}^{l}  \sum_{i_2=1}^{l}\beta_{i_1}\|\y^{(i_2)}\|_2\frac{\sqrt{t}}{\p_0\q_0-\p_1\q_1}T_{1,3}^{(\p,\q)}.
\end{eqnarray}
 From (\ref{eq:liftgenAanal19i}) and (\ref{eq:genDanal25}), we have
\begin{eqnarray}\label{eq:cpt4}
-\Omega_1& = & (1-t)\beta^2 \lp \mE_{G,{\mathcal U}_2}\langle \|\x^{(i_1)}\|_2^2\|\y^{(i_2)}\|_2^2\rangle_{\gamma_{01}^{(1)}} +   (s-1)\mE_{G,{\mathcal U}_2}\langle \|\x^{(i_1)}\|_2^2(\y^{(p_2)})^T\y^{(i_2)}\rangle_{\gamma_{02}^{(1)}} \rp \nonumber \\
& & - (1-t)s\beta^2(1-\m_1)\mE_{G,{\mathcal U}_2}\langle \|\x^{(i_1)}\|_2\|\x^{(p_1)}\|_2(\y^{(p_2)})^T\y^{(i_2)} \rangle_{\gamma_{1}^{(1)}} \nonumber\\
& & -(1-t)\beta^2 \lp \mE_{G,{\mathcal U}_2}\langle \|\x^{(i_1)}\|_2^2\|\y^{(i_2)}\|_2^2\rangle_{\gamma_{01}^{(1)}} +   (s-1)\mE_{G,{\mathcal U}_2}\langle \|\x^{(i_1)}\|_2^2(\y^{(p_2)})^T\y^{(i_2)}\rangle_{\gamma_{02}^{(1)}} \rp \nonumber \\
& & +(1-t)s\beta^2(1-\m_1)\mE_{G,{\mathcal U}_2}\langle \|\x^{(i_1)}\|_2\|\x^{(p_1)}\|_2(\y^{(p_2)})^T\y^{(i_2)} \rangle_{\gamma_{1}^{(1)}}\nonumber \\
 &   &
  +(1-t)s\beta^2\m_1\mE_{G,{\mathcal U}_2} \langle \|\x^{(i_1)}\|_2\|\x^{(p_1)}\|_2(\y^{(p_2)})^T\y^{(i_2)} \rangle_{\gamma_{2}^{(1)}} \nonumber \\
& = &
  (1-t)s\beta^2\m_1\mE_{G,{\mathcal U}_2} \langle \|\x^{(i_1)}\|_2\|\x^{(p_1)}\|_2(\y^{(p_2)})^T\y^{(i_2)} \rangle_{\gamma_{2}^{(1)}}.
\end{eqnarray}
From (\ref{eq:liftgenBanal20b}) and (\ref{eq:genEanal25}), we have
\begin{eqnarray}\label{eq:cpt5}
-\Omega_2 & = & (1-t)\beta^2 \lp\mE_{G,{\mathcal U}_2}\langle \|\x^{(i_1)}\|_2^2\|\y^{(i_2)}\|_2^2\rangle_{\gamma_{01}^{(1)}} +   (s-1)\mE_{G,{\mathcal U}_2}\langle \|\x^{(i_1)}\|_2^2 \|\y^{(i_2)}\|_2\|\y^{(p_2)}\|_2\rangle_{\gamma_{02}^{(1)}}\rp\nonumber \\
& & - (1-t) s\beta^2(1-\m_1)\mE_{G,{\mathcal U}_2}\langle (\x^{(p_1)})^T\x^{(i_1)}\|\y^{(i_2)}\|_2\|\y^{(p_2)}\|_2 \rangle_{\gamma_{1}^{(1)}}\nonumber \\
&  & -
(1-t)\beta^2\lp\mE_{G,{\mathcal U}_2}\langle \|\x^{(i_1)}\|_2^2\|\y^{(i_2)}\|_2^2\rangle_{\gamma_{01}^{(1)}} +   (s-1)\mE_{G,{\mathcal U}_2}\langle \|\x^{(i_1)}\|_2^2 \|\y^{(i_2)}\|_2\|\y^{(p_2)}\|_2\rangle_{\gamma_{02}^{(1)}}\rp\nonumber \\
& & + (1-t) s\beta^2(1-\m_1)\mE_{G,{\mathcal U}_2}\langle (\x^{(p_1)})^T\x^{(i_1)}\|\y^{(i_2)}\|_2\|\y^{(p_2)}\|_2 \rangle_{\gamma_{1}^{(1)}} \nonumber \\
&  & + (1-t) s\beta^2 \m_1\mE_{G,{\mathcal U}_2} \langle \|\y^{(i_2)}\|_2\|\y^{(p_2)}\|_2(\x^{(i_1)})^T\x^{(p_1)}\rangle_{\gamma_{2}^{(1)}} \nonumber \\
&  = &
   (1-t) s\beta^2\m_1\mE_{G,{\mathcal U}_2} \langle \|\y^{(i_2)}\|_2\|\y^{(p_2)}\|_2(\x^{(i_1)})^T\x^{(p_1)}\rangle_{\gamma_{2}^{(1)}}.
\end{eqnarray}
From (\ref{eq:liftgenCanal21b}) and (\ref{eq:genFanal29}), we have
  \begin{eqnarray}\label{eq:cpt6}
-\Omega_3 & = & t \beta^2 \lp \mE_{G,{\mathcal U}_2}\langle \|\x^{(i_1)}\|_2^2\|\y^{(i_2)}\|_2^2\rangle_{\gamma_{01}^{(1)}} +   (s-1)\mE_{G,{\mathcal U}_2}\langle \|\x^{(i_1)}\|_2^2 \|\y^{(i_2)}\|_2\|\y^{(p_2)}\|_2\rangle_{\gamma_{02}^{(1)}}\rp\nonumber \\
& & - ts\beta^2(1-\m_1)\mE_{G,{\mathcal U}_2}\langle \|\x^{(i_1)}\|_2\|\x^{(p_`)}\|_2\|\y^{(i_2)}\|_2\|\y^{(p_2)}\|_2 \rangle_{\gamma_{1}^{(1)}}\nonumber \\
&  & -
t\beta^2\lp\mE_{G,{\mathcal U}_2}\langle \|\x^{(i_1)}\|_2^2\|\y^{(i_2)}\|_2^2\rangle_{\gamma_{01}^{(1)}} +   (s-1)\mE_{G,{\mathcal U}_2}\langle \|\x^{(i_1)}\|_2^2 \|\y^{(i_2)}\|_2\|\y^{(p_2)}\|_2\rangle_{\gamma_{02}^{(1)}}\rp\nonumber \\
& & + t s\beta^2(1-\m_1)\mE_{G,{\mathcal U}_2}\langle \|\x^{(i_1)}\|_2\|\x^{(p_`)}\|_2\|\y^{(i_2)}\|_2\|\y^{(p_2)}\|_2 \rangle_{\gamma_{1}^{(1)}} \nonumber \\
&  & ts\beta^2\m_1\mE_{G,{\mathcal U}_2} \mE_{G,{\mathcal U}_2}\langle\|\x^{(i_2)}\|_2\|\x^{(p_2)}\|_2\|\y^{(i_2)}\|_2\|\y^{(p_2)}\rangle_{\gamma_{2}^{(1)}} \nonumber \\
& = &
 ts\beta^2
\m_1\mE_{G,{\mathcal U}_2} \langle\|\x^{(i_1)}\|_2\|\x^{(p_1)}\|_2\|\y^{(i_2)}\|_2\|\y^{(p_2)}\rangle_{\gamma_{2}^{(1)}}.
\end{eqnarray}
Finally, combining (\ref{eq:ctp1}) with (\ref{eq:cpt4})-(\ref{eq:cpt6}), we obtain
\begin{eqnarray}\label{eq:cpt7}
\frac{d\psi(\calX,\calY,\q,\m,\beta,s,t)}{d\p_1}  & = &       \frac{\mbox{sign}(s)\beta}{2\sqrt{n}} \phi^{(1,\p)}\nonumber \\
\frac{d\psi(\calX,\calY,\q,\m,\beta,s,t)}{d\q_1}  & = &       \frac{\mbox{sign}(s)\beta}{2\sqrt{n}} \phi^{(1,\q)},
 \end{eqnarray}
where
\begin{eqnarray}\label{eq:cpt8}
\phi^{(1,\p)} & = &
  -(1-t)s\beta^2\m_1\mE_{G,{\mathcal U}_2} \langle \|\x^{(i_1)}\|_2\|\x^{(p_1)}\|_2(\y^{(p_2)})^T\y^{(i_2)} \rangle_{\gamma_{2}^{(1)}} \nonumber \\
& &   - ts\beta^2\q_1
\m_1\mE_{G,{\mathcal U}_2} \langle\|\x^{(i_1)}\|_2\|\x^{(p_1)}\|_2\|\y^{(i_2)}\|_2\|\y^{(p_2)}\rangle_{\gamma_{2}^{(1)}}\nonumber \\
\phi^{(1,\q)} & = &
   -(1-t) s\beta^2\m_1\mE_{G,{\mathcal U}_2} \langle \|\y^{(i_2)}\|_2\|\y^{(p_2)}\|_2(\x^{(i_1)})^T\x^{(p_1)}\rangle_{\gamma_{2}^{(1)}} \nonumber \\
& &   - ts\beta^2\p_1
\m_1\mE_{G,{\mathcal U}_2} \langle\|\x^{(i_1)}\|_2\|\x^{(p_1)}\|_2\|\y^{(i_2)}\|_2\|\y^{(p_2)}\rangle_{\gamma_{2}^{(1)}}.
\end{eqnarray}

We summarize the above into the following proposition.
\begin{proposition}
\label{thm:thm1} Consider scalar $\m_1$, vector $\p=[\p_0,\p_1,\p_2]$ with $\p_0\geq \p_1\geq \p_2=0$,  and vector $\q=[\q_0,\q_1,\q_2]$ with $\q_0\geq \q_1\geq \q_2=0$. Let $k\in\{1,2\}$ and $G\in\mR^{m \times n},u^{(4,k)}\in\mR^1,\u^{(2,k)}\in\mR^{m\times 1}$, and $\h^{(k)}\in\mR^{n\times 1}$ all have i.i.d. standard normal components (they are then independent of each other as well). Let $a_k=\sqrt{\p_{k-1}\q_{k-1}-\p_k\q_k}$, $b_k=\sqrt{\p_{k-1}-\p_k}$, $c_k=\sqrt{\q_{k-1}-\q_k}$, and  ${\mathcal U}_k=[u^{(4,k)},\u^{(2,k)},\h^{(k)}]$. Assume that set ${\mathcal X}=\{\x^{(1)},\x^{(2)},\dots,\x^{(l)}\}$, where $\x^{(i)}\in \mR^{n},1\leq i\leq l$, and set ${\mathcal Y}=\{\y^{(1)},\y^{(2)},\dots,\y^{(l)}\}$, where $\y^{(i)}\in \mR^{m},1\leq i\leq l$ are given and that $\beta\geq 0$ and $s$ are real numbers. Consider the following function
\begin{equation}\label{eq:prop1eq1}
\psi(\calX,\calY,\p,\q,\m,\beta,s,t)  =  \mE_{G,{\mathcal U}_2} \frac{1}{\beta|s|\sqrt{n}\m_1} \log \mE_{{\mathcal U}_1} \lp \sum_{i_1=1}^{l}\lp\sum_{i_2=1}^{l}e^{\beta D_0^{(i_1,i_2)}} \rp^{s}\rp^{\m_1},
\end{equation}
where
\begin{eqnarray}\label{eq:prop1eq2}
 D_0^{(i_1,i_2)} & \triangleq & \sqrt{t}(\y^{(i_2)})^T
 G\x^{(i_1)}+\sqrt{1-t}\|\x^{(i_2)}\|_2 (\y^{(i_2)})^T(b_2\u^{(2,1)}+b_2\u^{(2,2)})\nonumber \\
 & & +\sqrt{t}\|\x^{(i_1)}\|_2\|\y^{(i_2)}\|_2(a_1u^{(4,1)}+a_2u^{(4,2)}) +\sqrt{1-t}\|\y^{(i_2)}\|_2(c_1\h^{(1)}+c_2\h^{(2)})^T\x^{(i)}.
 \end{eqnarray}
Then
\begin{eqnarray}\label{eq:prop1eq3}
\frac{d\psi(\calX,\calY,\q,\m,\beta,s,t)}{d\p_1}  & = &       \frac{\mbox{sign}(s)\beta}{2\sqrt{n}} \phi^{(1,\p)}\nonumber \\
\frac{d\psi(\calX,\calY,\q,\m,\beta,s,t)}{d\q_1}  & = &       \frac{\mbox{sign}(s)\beta}{2\sqrt{n}} \phi^{(1,\q)},
 \end{eqnarray}
where $\phi$'s are as in (\ref{eq:cpt8}).
\end{proposition}
\begin{proof}
  Follows from the above presentation.
\end{proof}

\section{$\p,\q$-derivatives at the $r$-th level of full lifting}
\label{sec:rthlev}

In this section, we generalize the results of Section \ref{sec:gencon} so that they hold for an arbitrary level, $r\in\mN$,  of full lifting.  All key conceptual steps needed for such a generalization are presented in Section \ref{sec:gencon}. Here, we formalize all the underlying technical details that ultimately allow utilizing the procedure from Section \ref{sec:gencon} for any level of lifting, $r\in\mN$.

We now consider general $r$ vectors $\m=[\m_1,\m_2,...,\m_r]$,
$\p=[\p_0,\p_1,...,\p_r,\p_{r+1}]$ with $\p_0\geq \p_1\geq \p_2\geq \dots \geq \p_r\geq\p_{r+1} = 0$ and $\q=[\q_0,\q_1,\q_2,\dots,\q_r,\q_{r+1}]$ with $\q_0\geq \q_1\geq \q_2\geq \dots \geq \q_r\geq \q_{r+1} = 0$. As earlier, we take $G\in\mR^{m \times n}$ and, for any $k\in\{1,2,\dots,r+1\}$,  $u^{(4,k)}\in\mR^1,\u^{(2,k)}\in\mR^{m\times 1}$, and $\h^{(k)}\in\mR^{n\times 1}$ to all have i.i.d. standard normal components (they are then independent of each other as well). $a_k=\sqrt{\p_{k-1}\q_{k-1}-\p_k\q_k}$, $b_k=\sqrt{\p_{k-1}-\p_k}$, $c_k=\sqrt{\q_{k-1}-\q_k}$ and ${\mathcal U}_k\triangleq [u^{(4,k)},\u^{(2,k)},\h^{(k)}]$.  Assuming that set ${\mathcal X}=\{\x^{(1)},\x^{(2)},\dots,\x^{(l)}\}$, where $\x^{(i)}\in \mR^{n},1\leq i\leq l$, and set ${\mathcal Y}=\{\y^{(1)},\y^{(2)},\dots,\y^{(l)}\}$, where $\y^{(i)}\in \mR^{m},1\leq i\leq l$ are given and that $\beta\geq 0$ and $s$ are real numbers, we are interested in the following function
\begin{equation}\label{eq:rthlev2genanal3}
\psi(\calX,\calY,\p,\q,\m,\beta,s,t)  =  \mE_{G,{\mathcal U}_{r+1}} \frac{1}{\beta|s|\sqrt{n}\m_r} \log
\lp \mE_{{\mathcal U}_{r}} \lp \dots \lp \mE_{{\mathcal U}_2}\lp\lp\mE_{{\mathcal U}_1} \lp Z^{\m_1}\rp\rp^{\frac{\m_2}{\m_1}}\rp\rp^{\frac{\m_3}{\m_2}} \dots \rp^{\frac{\m_{r}}{\m_{r-1}}}\rp,
\end{equation}
where
\begin{eqnarray}\label{eq:rthlev2genanal3a}
Z & \triangleq & \sum_{i_1=1}^{l}\lp\sum_{i_2=1}^{l}e^{\beta D_0^{(i_1,i_2)}} \rp^{s} \nonumber \\
 D_0^{(i_1,i_2)} & \triangleq & \sqrt{t}(\y^{(i_2)})^T
 G\x^{(i_1)}+\sqrt{1-t}\|\x^{(i_1)}\|_2 (\y^{(i_2)})^T\lp\sum_{k=1}^{r+1}b_k\u^{(2,k)}\rp\nonumber \\
 & & +\sqrt{t}\|\x^{(i_1)}\|_2\|\y^{(i_2)}\|_2\lp\sum_{k=1}^{r+1}a_ku^{(4,k)}\rp +\sqrt{1-t}\|\y^{(i_2)}\|_2\lp\sum_{k=1}^{r+1}c_k\h^{(k)}\rp^T\x^{(i_1)}
 \end{eqnarray}
To ensure the convenience of the exposition, we, as earlier and analogously to (\ref{eq:genanal4}), set
\begin{eqnarray}\label{eq:rthlev2genanal4}
\u^{(i_1,1)} & =  & \frac{G\x^{(i_1)}}{\|\x^{(i_1)}\|_2} \nonumber \\
\u^{(i_1,3,k)} & =  & \frac{(\h^{(k)})^T\x^{(i_1)}}{\|\x^{(i_1)}\|_2},
\end{eqnarray}
and recall that
\begin{eqnarray}\label{eq:rthlev2genanal5}
\u_j^{(i_1,1)} & =  & \frac{G_{j,1:n}\x^{(i_1)}}{\|\x^{(i_1)}\|_2},1\leq j\leq m.
\end{eqnarray}
It is relatively easy to see that, for any fixed $i_1$, the elements of $\u^{(i_1,1)}$, $\u^{(2,k)}$, and $\u^{(i_1,3,k)}$ are standard normals.   (\ref{eq:rthlev2genanal3}) can then be rewritten as
\begin{equation}\label{eq:rthlev2genanal6}
\psi(\calX,\calY,\p,\q,\m,\beta,s,t)  =  \mE_{G,{\mathcal U}_{r+1}} \frac{1}{\beta|s|\sqrt{n}\m_r} \log
\lp \mE_{{\mathcal U}_{r}} \lp \dots \lp \mE_{{\mathcal U}_2}\lp\lp\mE_{{\mathcal U}_1} \lp Z^{\m_1}\rp\rp^{\frac{\m_2}{\m_1}}\rp\rp^{\frac{\m_3}{\m_2}} \dots \rp^{\frac{\m_{r}}{\m_{r-1}}}\rp,
\end{equation}
where $\beta_{i_1}=\beta\|\x^{(i_1)}\|_2$ and now
\begin{eqnarray}\label{eq:rthlev2genanal7}
B^{(i_1,i_2)} & \triangleq &  \sqrt{t}(\y^{(i_2)})^T\u^{(i_1,1)}+\sqrt{1-t} (\y^{(i_2)})^T\lp \sum_{k=1}^{r+1}b_k\u^{(2,k)} \rp \nonumber \\
D^{(i_1,i_2)} & \triangleq &  B^{(i_1,i_2)}+\sqrt{t}\|\y^{(i_2)}\|_2 \lp \sum_{k=1}^{r+1}a_ku^{(4,k)}\rp+\sqrt{1-t}\|\y^{(i_2)}\|_2 \lp \sum_{k=1}^{r+1}c_k\u^{(i_1,3,k)}  \rp   \nonumber \\
A^{(i_1,i_2)} & \triangleq &  e^{\beta_{i_1}D^{(i_1,i_2)}}\nonumber \\
C^{(i_1)} & \triangleq &  \sum_{i_2=1}^{l}A^{(i_1,i_2)}\nonumber \\
Z & \triangleq & \sum_{i_1=1}^{l} \lp \sum_{i_2=1}^{l} A^{(i_1,i_2)}\rp^s =\sum_{i_1=1}^{l}  (C^{(i_1)})^s.
\end{eqnarray}
We set $\m_0=1$ and
\begin{eqnarray}\label{eq:rthlev2genanal7a}
\zeta_r\triangleq \mE_{{\mathcal U}_{r}} \lp \dots \lp \mE_{{\mathcal U}_2}\lp\lp\mE_{{\mathcal U}_1} \lp Z^{\frac{\m_1}{\m_0}}\rp\rp^{\frac{\m_2}{\m_1}}\rp\rp^{\frac{\m_3}{\m_2}} \dots \rp^{\frac{\m_{r}}{\m_{r-1}}}, r\geq 1.
\end{eqnarray}
One can then write
\begin{eqnarray}\label{eq:rthlev2genanal7b}
\zeta_k = \mE_{{\mathcal U}_{k}} \lp  \zeta_{k-1} \rp^{\frac{\m_{k}}{\m_{k-1}}}, k\geq 2,\quad \mbox{and} \quad
\zeta_1=\mE_{{\mathcal U}_1} \lp Z^{\frac{\m_1}{\m_0}}\rp.
\end{eqnarray}
For the completeness, we also set $\zeta_0=Z$ and further write for the derivative
\begin{eqnarray}\label{eq:rthlev2genanal9}
\frac{d\psi(\calX,\calY,\q,\m,\beta,s,t)}{d\p_{k_1}}
& = &  \frac{d}{dt}\lp\mE_{G,{\mathcal U}_{r+1}} \frac{1}{\beta|s|\sqrt{n}\m_r} \log \lp \zeta_r\rp \rp \nonumber \\
& = &  \mE_{G,{\mathcal U}_{r+1}} \frac{1}{\beta|s|\sqrt{n}\m_r\zeta_r} \frac{d\zeta_r}{d\p_{k_1}} \nonumber \\
& = &  \mE_{G,{\mathcal U}_{r+1}} \frac{1}{\beta|s|\sqrt{n}\m_{r-1}\zeta_r} \mE_{{\mathcal U}_{r}} \zeta_{r-1}^{\frac{\m_r}{\m_{r-1}}-1} \frac{d\zeta_{r-1}}{d\p_{k_1}} \nonumber \\
& = &  \mE_{G,{\mathcal U}_{r+1}} \frac{1}{\beta|s|\sqrt{n}\m_{r-2}\zeta_r} \mE_{{\mathcal U}_{r}} \zeta_{r-1}^{\frac{\m_r}{\m_{r-1}}-1}
 \mE_{{\mathcal U}_{r-1}} \zeta_{r-2}^{\frac{\m_{r-1}}{\m_{r-2}}-1}
\frac{d\zeta_{r-2}}{d\p_{k_1}} \nonumber \\
& = &  \mE_{G,{\mathcal U}_{r+1}} \frac{1}{\beta|s|\sqrt{n}\m_1\zeta_r}
\prod_{k=r}^{2}\mE_{{\mathcal U}_{k}} \zeta_{k-1}^{\frac{\m_k}{\m_{k-1}}-1}
 \frac{d\zeta_{1}}{dt} \nonumber \\
& = &  \mE_{G,{\mathcal U}_{r+1}} \frac{1}{\beta|s|\sqrt{n}\m_1\zeta_r}
\prod_{k=r}^{2}\mE_{{\mathcal U}_{k}} \zeta_{k-1}^{\frac{\m_k}{\m_{k-1}}-1}
 \frac{d \mE_{{\mathcal U}_1} Z^{\m_1} }{d\p_{k_1}} \nonumber \\
 & = &
\mE_{G,{\mathcal U}_{r+1}} \frac{1}{\beta|s|\sqrt{n}\zeta_r}
\prod_{k=r}^{2}\mE_{{\mathcal U}_{k}} \zeta_{k-1}^{\frac{\m_k}{\m_{k-1}}-1}
 \mE_{{\mathcal U}_1} \frac{1}{Z^{1-\m_1}}  \sum_{i=1}^{l} (C^{(i_1)})^{s-1} \nonumber \\
& & \times \sum_{i_2=1}^{l}\beta_{i_1}A^{(i_1,i_2)}\frac{dD^{(i_1,i_2)}}{d\p_{k_1}},
\end{eqnarray}
where the product runs in an \emph{index decreasing order} and
\begin{eqnarray}\label{eq:rthlev2genanal9a}
\frac{dD^{(i_1,i_2)}}{d\p_{k_1}} & = & \lp \frac{dB^{(i_1,i_2)}}{d\p_{k_1}}+
\sqrt{t} \frac{d \lp\|\y^{(i_2)}\|_2 (\sum_{k=k_1}^{k_1+1} a_ku^{(4,k)}) \rp}{d\p_{k_1}}\rp \nonumber \\
& = & \lp \frac{dB^{(i_1,i_2)}}{d\p_{k_1}}
-\sqrt{t} \q_{k_1}\frac{\|\y^{(i_2)}\|_2 u^{(4,k_1)}}{2\sqrt{\p_{k_1-1}\q_{k_1-1}-\p_{k_1}\q_{k_1}}}+\sqrt{t} \q_{k_1}\frac{\|\y^{(i_2)}\|_2 u^{(4,k_1+1)}}{2\sqrt{\p_{k_1}\q_{k_1}-\p_{k_1+1}\q_{k_1+1}}}\rp.
\end{eqnarray}
Utilization of (\ref{eq:rthlev2genanal7}) gives
\begin{eqnarray}\label{eq:rthlev2genanal10}
\frac{dB^{(i_1,i_2)}}{d\p_{k_1}} & = &   \frac{d\lp\sqrt{t}(\y^{(i_2)})^T\u^{(i_1,1)}+\sqrt{1-t} (\y^{(i_2)})^T(\sum_{k=1}^{r+1} b_k\u^{(2,k)})\rp}{d\p_{k_1}} \nonumber \\
& = &   \frac{d\lp\sqrt{1-t} (\y^{(i_2)})^T(\sum_{k=k_1}^{k_1+1} b_k\u^{(2,k)})\rp}{d\p_{k_1}} \nonumber \\
 & = &
-\sqrt{1-t}\sum_{j=1}^{m}\lp  \frac{\y_j^{(i_2)} \u_j^{(2,k_1)}}{2\sqrt{\p_{k_1-1}-\p_{k_1}}}\rp
+\sqrt{1-t}\sum_{j=1}^{m}\lp  \frac{\y_j^{(i_2)} \u_j^{(2,k_1+1)}}{2\sqrt{\p_{k_1}-\p_{k_1+1}}}\rp.
\end{eqnarray}
Writing analogously to (\ref{eq:genanal10e}), we have
\begin{equation}\label{eq:rthlev2genanal10e}
\frac{d\psi(\calX,\calY,\p,\q,\m,\beta,s,t)}{d\p_{k_1}}  =       \frac{\mbox{sign}(s)}{2\beta\sqrt{n}} \sum_{i_1=1}^{l}  \sum_{i_2=1}^{l}
\beta_{i_1}\lp T_{k_1+1}^{\p}-T_{k_1}^{\p}\rp,
\end{equation}
where
\begin{eqnarray}\label{eq:rthlev2genanal10f}
 T_k^{\p} & = & \frac{\sqrt{1-t}}{\sqrt{\p_{k-1}-\p_{k}}}\sum_{j=1}^{m}T_{k,1,j}^{\p}
 +\frac{\sqrt{t}\q_{k_1}\|\y^{(i_2)}\|_2}{\sqrt{\p_{k-1}\q_{k-1}-\p_{k}\q_{k}}}T_{k,3}^{(\p,\q)}, k\in\{k_1,k_1+1\},
 \end{eqnarray}
 and
 \begin{eqnarray}\label{eq:rthlev2genanal10g}
 T_{k,1,j}^{\p} & = &   \mE_{G,{\mathcal U}_{r+1}} \lp
\zeta_r^{-1}\prod_{v=r}^{2}\mE_{{\mathcal U}_{v}} \zeta_{v-1}^{\frac{\m_v}{\m_{v-1}}-1}
  \mE_{{\mathcal U}_1}\frac{(C^{(i_1)})^{s-1} A^{(i_1,i_2)} \y_j^{(i_2)}\u_j^{(2,k)}}{Z^{1-\m_1}} \rp \nonumber \\
T_{k,3}^{(\p,\q)} & = &  \mE_{G,{\mathcal U}_{r+1}} \lp
\zeta_r^{-1}\prod_{v=r}^{2}\mE_{{\mathcal U}_{v}} \zeta_{v-1}^{\frac{\m_v}{\m_{v-1}}-1}
  \mE_{{\mathcal U}_1}\frac{(C^{(i_1)})^{s-1} A^{(i_1,i_2)} u^{(4,k)}}{Z^{1-\m_1}} \rp.
\end{eqnarray}
Keeping in mind the results for the first level of full lifting one can immediately write analogously to (\ref{eq:rthlev2genanal10e})-(\ref{eq:rthlev2genanal10g})
\begin{equation}\label{eq:rthlev2genanal10e1}
\frac{d\psi(\calX,\calY,\p,\q,\m,\beta,s,t)}{d\q_{k_1}}  =       \frac{\mbox{sign}(s)}{2\beta\sqrt{n}} \sum_{i_1=1}^{l}  \sum_{i_2=1}^{l}
\beta_{i_1}\lp T_{k_1+1}^{\q}-T_{k_1}^{\q}\rp,
\end{equation}
where
\begin{eqnarray}\label{eq:rthlev2genanal10f1}
 T_k^{\q} & = & \frac{\sqrt{1-t}}{\sqrt{\q_{k-1}-\q_{k}}} \|\y^{(i_2)}\|_2 T_{k,2}^{\q}
 +\frac{\sqrt{t}\p_{k_1}\|\y^{(i_2)}\|_2}{\sqrt{\p_{k-1}\q_{k-1}-\p_{k}\q_{k}}}T_{k,3}^{(\p,\q)}, k\in\{k_1,k_1+1\},
 \end{eqnarray}
 and
 \begin{eqnarray}\label{eq:rthlev2genanal10g1}
T_{k,2}^{\p} & = &   \mE_{G,{\mathcal U}_{r+1}} \lp
\zeta_r^{-1}\prod_{v=r}^{2}\mE_{{\mathcal U}_{v}} \zeta_{v-1}^{\frac{\m_v}{\m_{v-1}}-1}
  \mE_{{\mathcal U}_1}\frac{(C^{(i_1)})^{s-1} A^{(i_1,i_2)} \u^{(i_1,3,k)}}{Z^{1-\m_1}} \rp \nonumber \\
T_{k,3}^{(\p,\q)} & = &  \mE_{G,{\mathcal U}_{r+1}} \lp
\zeta_r^{-1}\prod_{v=r}^{2}\mE_{{\mathcal U}_{v}} \zeta_{v-1}^{\frac{\m_v}{\m_{v-1}}-1}
  \mE_{{\mathcal U}_1}\frac{(C^{(i_1)})^{s-1} A^{(i_1,i_2)} u^{(4,k)}}{Z^{1-\m_1}} \rp.
\end{eqnarray}
We clearly observe three sequences $\lp T_{k,1,j}^{\p}\rp_{k=1:r+1}$, $\lp T_{k,2}^{\q}\rp_{k=1:r+1}$, and $\lp T_{k,3}^{(\p,\q)}\rp_{k=1:r+1}$. Similarly to what was observed in \cite{Stojnicnflgscompyx23}, for all our purposes, the internal connections among the components within each of these three sequences are identical. We will describe how to handle one of them, $\lp T_{k,1,j}^{\p}\rp_{k=1:r+1}$. For the remaining two, we then quickly deduce the final results.

\subsection{Scaling + reweightedly averaged new terms}
\label{sec:scaledcanout}

As we below formalize the relation between successive sequence elements, $T_{k_1,1,j}$ and $T_{k_1+1,1,j}$ (for $k_1\geq 2$), it will be clear that there are two key parts of such a recursive/inductive relation: 1) \emph{successive scaling}; and 2) appearance of an appropriately \emph{reweightedly averaged (over a $\gamma$ measure) new term}. A similar concept was observed in \cite{Stojnicnflgscompyx23} as well.

We start by writing
 \begin{eqnarray}\label{eq:rthlev2genanal11}
 T_{k_1,1,j}^{\p} & = &   \mE_{G,{\mathcal U}_{r+1}} \lp
\zeta_r^{-1}\prod_{v=r}^{2}\mE_{{\mathcal U}_{v}} \zeta_{v-1}^{\frac{\m_v}{\m_{v-1}}-1}
  \mE_{{\mathcal U}_1}\frac{(C^{(i_1)})^{s-1} A^{(i_1,i_2)} \y_j^{(i_2)}\u_j^{(2,k_1)}}{Z^{1-\m_1}} \rp \nonumber \\
 & = &   \mE_{G,{\mathcal U}_{r+1}} \lp
\zeta_r^{-1}\prod_{v=r}^{k_1+1}\mE_{{\mathcal U}_{v}} \zeta_{v-1}^{\frac{\m_v}{\m_{v-1}}-1}
\prod_{v=k_1}^{2}\mE_{{\mathcal U}_{v}} \zeta_{v-1}^{\frac{\m_v}{\m_{v-1}}-1}
  \mE_{{\mathcal U}_1}\frac{(C^{(i_1)})^{s-1} A^{(i_1,i_2)} \y_j^{(i_2)}\u_j^{(2,k_1)}}{Z^{1-\m_1}} \rp \nonumber \\
    & = &   \mE_{G,{\mathcal U}_{r+1}} \Bigg(\Bigg.
\zeta_r^{-1}\prod_{v=r}^{k_1+1}\mE_{{\mathcal U}_{v}} \zeta_{v-1}^{\frac{\m_v}{\m_{v-1}}-1}
\mE_{{\mathcal U}_{k_1}} \mE_{{\mathcal U}_{k_1}} \lp \u_j^{(2,k_1)}\u_j^{(2,k_1)}\rp \nonumber \\
& & \times
\frac{d}{d\u_j^{(2,k_1)}} \lp \zeta_{k_1-1}^{\frac{\m_{k_1}}{\m_{k_1-1}}-1}\prod_{v=k_1-1}^{2}\mE_{{\mathcal U}_{v}} \zeta_{v-1}^{\frac{\m_v}{\m_{v-1}}-1}
  \mE_{{\mathcal U}_1}\frac{(C^{(i_1)})^{s-1} A^{(i_1,i_2)} \y_j^{(i_2)}\u_j^{(2,k_1)}}{Z^{1-\m_1}} \rp \Bigg.\Bigg) \nonumber \\
    & = &  \frac{\sqrt{\p_{k_1-1}-\p_{k_1}}}{\p_{k_1-1}-\p_{k_1}} \nonumber \\
    & & \times \mE_{G,{\mathcal U}_{r+1}} \Bigg(\Bigg.
\zeta_r^{-1}\prod_{v=r}^{k_1+1}\mE_{{\mathcal U}_{v}} \zeta_{v-1}^{\frac{\m_v}{\m_{v-1}}-1}
\mE_{{\mathcal U}_{k_1}} \mE_{{\mathcal U}_{k_1}} \lp \sqrt{\p_{k_1-1}-\p_{k_1}}\u_j^{(2,k_1)}\sqrt{\p_{k_1-1}-\p_{k_1}}\u_j^{(2,k_1)}\rp \nonumber \\
& & \times
\frac{d}{d\lp \sqrt{\p_{k_1-1}-\p_{k_1}} \u_j^{(2,k_1)} \rp} \lp \zeta_{k_1-1}^{\frac{\m_{k_1}}{\m_{k_1-1}}-1}\prod_{v=k_1-1}^{2}\mE_{{\mathcal U}_{v}} \zeta_{v-1}^{\frac{\m_v}{\m_{v-1}}-1}
  \mE_{{\mathcal U}_1}\frac{(C^{(i_1)})^{s-1} A^{(i_1,i_2)} \y_j^{(i_2)}\u_j^{(2,k_1)}}{Z^{1-\m_1}} \rp \Bigg.\Bigg) \nonumber \\
   & = &  \frac{\sqrt{\p_{k_1-1}-\p_{k_1}}}{\p_{k_1-1}-\p_{k_1}}  T_{k_1,1,j},
\end{eqnarray}
where \cite{Stojnicnflgscompyx23} determined $T_{k_1,1,j}$.

Moreover,
 \begin{eqnarray}\label{eq:rthlev2genanal12}
 T_{k_1+1,1,j}^{\p}
  & = &   \mE_{G,{\mathcal U}_{r+1}} \lp
\zeta_r^{-1}\prod_{v=r}^{k_1+2}\mE_{{\mathcal U}_{v}} \zeta_{v-1}^{\frac{\m_v}{\m_{v-1}}-1}
\prod_{v=k_1+1}^{2}\mE_{{\mathcal U}_{v}} \zeta_{v-1}^{\frac{\m_v}{\m_{v-1}}-1}
  \mE_{{\mathcal U}_1}\frac{(C^{(i_1)})^{s-1} A^{(i_1,i_2)} \y_j^{(i_2)}\u_j^{(2,k_1+1)}}{Z^{1-\m_1}} \rp \nonumber \\
    & = &   \mE_{G,{\mathcal U}_{r+1}} \Bigg( \Bigg.
\zeta_r^{-1}\prod_{v=r}^{k_1+2}\mE_{{\mathcal U}_{v}} \zeta_{v-1}^{\frac{\m_v}{\m_{v-1}}-1}
\mE_{{\mathcal U}_{k_1+1}}   \nonumber \\
& & \times  \zeta_{k_1}^{\frac{\m_{k_1+1}}{\m_{k_1}}-1}
\prod_{v=k_1}^{2}\mE_{{\mathcal U}_{v}} \zeta_{v-1}^{\frac{\m_v}{\m_{v-1}}-1}
  \mE_{{\mathcal U}_1}\frac{(C^{(i_1)})^{s-1} A^{(i_1,i_2)} \y_j^{(i_2)}\u_j^{(2,k_1+1)}}{Z^{1-\m_1}} \Bigg. \Bigg) \nonumber \\
      & = &   \mE_{G,{\mathcal U}_{r+1}} \Bigg( \Bigg.
\zeta_r^{-1}\prod_{v=r}^{k_1+2}\mE_{{\mathcal U}_{v}} \zeta_{v-1}^{\frac{\m_v}{\m_{v-1}}-1}
\mE_{{\mathcal U}_{k_1+1}} \mE_{{\mathcal U}_{k_1+1}} (\u_j^{(2,k_1+1)}\u_j^{(2,k_1+1)}) \nonumber \\
& & \times  \frac{d}{d\u_j^{(2,k_1+1)}}  \lp \zeta_{k_1}^{\frac{\m_{k_1+1}}{\m_{k_1}}-1}
\prod_{v=k_1}^{2}\mE_{{\mathcal U}_{v}} \zeta_{v-1}^{\frac{\m_v}{\m_{v-1}}-1}
  \mE_{{\mathcal U}_1}\frac{(C^{(i_1)})^{s-1} A^{(i_1,i_2)} \y_j^{(i_2)}}{Z^{1-\m_1}} \rp   \Bigg. \Bigg)   \nonumber \\
      & = &  \frac{\sqrt{\p_{k_1} -\p_{k_1+1}}
}{\p_{k_1} -\p_{k_1+1}} \nonumber \\
& & \times \mE_{G,{\mathcal U}_{r+1}} \Bigg( \Bigg.
\zeta_r^{-1}\prod_{v=r}^{k_1+2}\mE_{{\mathcal U}_{v}} \zeta_{v-1}^{\frac{\m_v}{\m_{v-1}}-1}
\mE_{{\mathcal U}_{k_1+1}} \mE_{{\mathcal U}_{k_1+1}} (\sqrt{\p_{k_1} -\p_{k_1+1}}
\u_j^{(2,k_1+1)}\sqrt{\p_{k_1} -\p_{k_1+1}}
\u_j^{(2,k_1+1)}) \nonumber \\
& & \times  \frac{d}{d \lp \sqrt{\p_{k_1} -\p_{k_1+1}}
 \u_j^{(2,k_1+1)}\rp}  \lp \zeta_{k_1}^{\frac{\m_{k_1+1}}{\m_{k_1}}-1}
\prod_{v=k_1}^{2}\mE_{{\mathcal U}_{v}} \zeta_{v-1}^{\frac{\m_v}{\m_{v-1}}-1}
  \mE_{{\mathcal U}_1}\frac{(C^{(i_1)})^{s-1} A^{(i_1,i_2)} \y_j^{(i_2)}}{Z^{1-\m_1}} \rp   \Bigg. \Bigg)   \nonumber \\
      & = &  \frac{\sqrt{\p_{k_1} -\p_{k_1+1}}
}{\p_{k_1} -\p_{k_1+1}} T_{k_1+1,1,j} \nonumber \\
 & = &  \frac{\sqrt{\p_{k_1} -\p_{k_1+1}}
}{\p_{k_1} -\p_{k_1+1}}\lp \Pi_{k_1+1,1,j}^{(1)}
+  \Pi_{k_1+1,1,j}^{(2)}\rp,
 \end{eqnarray}
where all three $T_{k_1+1,1,j}$, $\Pi_{k_1+1,1,j}^{(1)}$, and $\Pi_{k_1+1,1,j}^{(1)}$ are determined in \cite{Stojnicnflgscompyx23}. In particular,  \cite{Stojnicnflgscompyx23} established
 \begin{eqnarray}\label{eq:rthlev2genanal19}
 \Pi_{k_1+1,1,j}^{(1)}=\frac{(\p_{k_1} -\p_{k_1+1})}{(\p_{k_1-1}-\p_{k_1})}T_{k_1,1,j},
 \end{eqnarray}
and (while recalling that $\zeta_0=Z$ and $\m_0=1$) also
\begin{eqnarray}\label{eq:rthlev2genanal26}
 \sum_{i_1=1}^{l}\sum_{i_2=1}^{l} \sum_{j=1}^{m}
\beta_{i_1} \Pi_{k_1+1,1,j}^{(2)}
  & = & -\sqrt{1-t} s\beta^2(\p_{k_1}-\p_{k_1+1}) \lp \m_{k_1} - \m_{k_1+1} \rp \nonumber\\
 & & \times
  \mE_{G,{\mathcal U}_{r+1}} \left \langle \|\x^{(i_1)}\|_2\|\x^{(p_1)}\|_2 (\y^{(p_2)})^T \y^{(i_2)} \right \rangle_{\gamma_{k_1+1}^{(r)}}.
\end{eqnarray}
Combining (\ref{eq:rthlev2genanal12}), (\ref{eq:rthlev2genanal19}), and (\ref{eq:rthlev2genanal26}), we find
 \begin{eqnarray}\label{eq:rthlev2genanal27}
 \sum_{i_1=1}^{l}\sum_{i_2=1}^{l} \sum_{j=1}^{m}
 \frac{\beta_{i_1}\sqrt{1-t}}{\sqrt{\p_{k_1}-\p_{k_1+1}}} T_{k_1+1,1,j}^{\p}
& = &
 \sum_{i_1=1}^{l}\sum_{i_2=1}^{l} \sum_{j=1}^{m}
\beta_{i_1} \frac{\sqrt{1-t}}{\p_{k_1}-\p_{k_1+1}}     \lp \Pi_{k_1+1,1,j}^{(1)}
+  \Pi_{k_1+1,1,j}^{(2)}\rp \nonumber \\
& = &
 \sum_{i_1=1}^{l}\sum_{i_2=1}^{l} \sum_{j=1}^{m}
 \frac{\beta_{i_1}\sqrt{1-t}T_{k_1,1,j}}{\p_{k_1-1}-\p_{k_1}}  +   \sum_{i_1=1}^{l}\sum_{i_2=1}^{l} \sum_{j=1}^{m}
 \frac{\beta_{i_1}\sqrt{1-t} \Pi_{k_1+1,1,j}^{(2)}}{\p_{k_1}-\p_{k_1+1}} \nonumber \\
& = &
 \sum_{i_1=1}^{l}\sum_{i_2=1}^{l} \sum_{j=1}^{m}
 \frac{\beta_{i_1}\sqrt{1-t}T_{k_1,1,j}^{\p}}{\sqrt{\p_{k_1-1}-\p_{k_1}}}  +   \sum_{i_1=1}^{l}\sum_{i_2=1}^{l} \sum_{j=1}^{m}
 \frac{\beta_{i_1}\sqrt{1-t} \Pi_{k_1+1,1,j}^{(2)}}{\p_{k_1}-\p_{k_1+1}} \nonumber \\
 & = & \sum_{i_1=1}^{l}\sum_{i_2=1}^{l} \sum_{j=1}^{m}
 \frac{\beta_{i_1}\sqrt{1-t}T_{k_1,1,j}^{\p}}{\sqrt{\p_{k_1-1}-\p_{k_1}}}   \nonumber \\
 &  & - (1-t)\Bigg( \Bigg. s\beta^2 \lp \m_{k_1} -  \m_{k_1+1} \rp  \nonumber \\
 & & \times
  \mE_{G,{\mathcal U}_{r+1}} \left \langle \|\x^{(i_1)}\|_2\|\x^{(p_1)}\|_2 (\y^{(p_2)})^T \y^{(i_2)} \right \rangle_{\gamma_{k_1+1}^{(r)} } \Bigg. \Bigg).
\end{eqnarray}
Repeating the above mechanism and keeping in mind the results from Section \ref{sec:gencon} related to the first level of full lifting, one can then analogously write for the remaining two sequences $\lp T_{k,2}^{\q}\rp_{k=1:r+1}$ and $\lp T_{k,3}^{(\p,\q)}\rp_{k=1:r+1}$
\begin{eqnarray}\label{eq:rthlev2genanal28}
 \sum_{i_1=1}^{l}\sum_{i_2=1}^{l}
 \frac{\beta_{i_1}\sqrt{1-t} \|\y^{(i_2)}\|_2}{\sqrt{\q_{k_1}-\q_{k_1+1}}} T_{k_1+1,2}^{\q}
 & = &  \sum_{i_1=1}^{l}\sum_{i_2=1}^{l} \sum_{j=1}^{m}
 \frac{\beta_{i_1}\sqrt{1-t} \|\y^{(i_2)}\|_2T_{k_1,1,j}^{\q}}{\sqrt{\q_{k_1-1}-\q_{k_1}}} \nonumber \\
 &  & - (1-t) \Bigg( \Bigg. s\beta^2 \lp \m_{k_1} - \m_{k_1+1} \rp \nonumber \\
 & &
\times  \mE_{G,{\mathcal U}_{r+1}} \left \langle   (\x^{(p_1)})^T \x^{(i_1)} \|\y^{(i_2)}\|_2\|\y^{(p_2)}\|_2  \right \rangle_{\gamma_{k_1+1}^{(r)} } \Bigg.\Bigg), \nonumber \\
\end{eqnarray}
and
 \begin{eqnarray}\label{eq:rthlev2genanal29}
 \sum_{i_1=1}^{l}\sum_{i_2=1}^{l}
 \frac{\beta_{i_1}\sqrt{t} \|\y^{(i_2)}\|_2}{\sqrt{\p_{k_1}\q_{k_1}-\p_{k_1+1}\q_{k_1+1}}} T_{k_1+1,3}^{(\p,\q)}
 & = &   \sum_{i_1=1}^{l}\sum_{i_2=1}^{l}
\frac{\beta_{i_1} \sqrt{t} \|\y^{(i_2)}\|_2}{\sqrt{\p_{k_1-1}\q_{k_1-1}-\p_{k_1}\q_{k_1}}} T_{k_1,3}^{(\p,\q)}
\nonumber \\
 &  & - t\Bigg( \Bigg. s\beta^2 \lp  \m_{k_1} - \m_{k_1+1} \rp \nonumber \\
& & \times   \mE_{G,{\mathcal U}_{r+1}} \left \langle   \|\x^{(i_1)}\|_2\|\x^{(p_1)}\|_2 \|\y^{(i_2)}\|_2\|\y^{(p_2)}\|_2  \right \rangle_{\gamma_{k_1+1}^{(r)} } \Bigg. \Bigg).
\end{eqnarray}
One now observes the \emph{scaled canceling out} mechanism similar to the one observed in \cite{Stojnicnflgscompyx23}. Basically, adding the succesive elements in each of these three sequences, as suggested in (\ref{eq:rthlev2genanal10e})-(\ref{eq:rthlev2genanal10f}) and (\ref{eq:rthlev2genanal10e1})-(\ref{eq:rthlev2genanal10f1}), results in canceling out parts of the first and second summands.

\subsection{Connecting all pieces together}
\label{sec:rthall}

We can now summarize the above discussion into the following theorem.
\begin{theorem}
\label{thm:thm3}
Let $r\in\mN$ and $k\in\{1,2,\dots,r+1\}$ and consider vectors $\m=[\m_0,\m_1,\m_2,...,\m_r,\m_{r+1}]$ with $\m_0=1$ and $\m_{r+1}=0$,
$\p=[\p_0,\p_1,...,\p_r,\p_{r+1}]$ with $1\geq\p_0\geq \p_1\geq \p_2\geq \dots \geq \p_r\geq \p_{r+1} =0$ and $\q=[\q_0,\q_1,\q_2,\dots,\q_r,\q_{r+1}]$ with $1\geq\q_0\geq \q_1\geq \q_2\geq \dots \geq \q_r\geq \q_{r+1} = 0$. Let the components of $G\in\mR^{m\times n}$, $u^{(4,k)}\in\mR$, $\u^{(2,k)}\in\mR^m$, and $\h^{(k)}\in\mR^n$ be i.i.d. standard normals. Also, let $a_k=\sqrt{\p_{k-1}\q_{k-1}-\p_k\q_k}$, $b_k=\sqrt{\p_{k-1}-\p_{k}}$, and $c_k=\sqrt{\q_{k-1}-\q_{k}}$, and let ${\mathcal U}_k\triangleq [u^{(4,k)},\u^{(2,k)},\h^{(2k)}]$.  Assuming that set ${\mathcal X}=\{\x^{(1)},\x^{(2)},\dots,\x^{(l)}\}$, where $\x^{(i)}\in \mR^{n},1\leq i\leq l$, set ${\mathcal Y}=\{\y^{(1)},\y^{(2)},\dots,\y^{(l)}\}$, where $\y^{(i)}\in \mR^{m},1\leq i\leq l$, and scalars $\beta\geq 0$ and $s\in\mR$ are given, and consider the following function
\begin{equation}\label{eq:thm3eq1}
\psi(\calX,\calY,\p,\q,\m,\beta,s,t)  =  \mE_{G,{\mathcal U}_{r+1}} \frac{1}{\beta|s|\sqrt{n}\m_r} \log
\lp \mE_{{\mathcal U}_{r}} \lp \dots \lp \mE_{{\mathcal U}_2}\lp\lp\mE_{{\mathcal U}_1} \lp Z^{\m_1}\rp\rp^{\frac{\m_2}{\m_1}}\rp\rp^{\frac{\m_3}{\m_2}} \dots \rp^{\frac{\m_{r}}{\m_{r-1}}}\rp,
\end{equation}
where
\begin{eqnarray}\label{eq:thm3eq2}
Z & \triangleq & \sum_{i_1=1}^{l}\lp\sum_{i_2=1}^{l}e^{\beta D_0^{(i_1,i_2)}} \rp^{s} \nonumber \\
 D_0^{(i_1,i_2)} & \triangleq & \sqrt{t}(\y^{(i_2)})^T
 G\x^{(i_1)}+\sqrt{1-t}\|\x^{(i_2)}\|_2 (\y^{(i_2)})^T\lp\sum_{k=1}^{r+1}b_k\u^{(2,k)}\rp\nonumber \\
 & & +\sqrt{t}\|\x^{(i_1)}\|_2\|\y^{(i_2)}\|_2\lp\sum_{k=1}^{r+1}a_ku^{(4,k)}\rp +\sqrt{1-t}\|\y^{(i_2)}\|_2\lp\sum_{k=1}^{r+1}c_k\h^{(k)}\rp^T\x^{(i)}
 \end{eqnarray}
Let $\zeta_0=Z$ and let $\zeta_k,k\geq 1$, be as defined in (\ref{eq:rthlev2genanal7a}) and (\ref{eq:rthlev2genanal7b}). Moreover, consider the operators
\begin{eqnarray}\label{eq:thm3eq3}
 \Phi_{{\mathcal U}_k} & \triangleq &  \mE_{{\mathcal U}_{k}} \frac{\zeta_{k-1}^{\frac{\m_k}{\m_{k-1}}}}{\zeta_k},
 \end{eqnarray}
and set
\begin{eqnarray}\label{eq:thm3eq4}
  \gamma_0(i_1,i_2) & = &
\frac{(C^{(i_1)})^{s}}{Z}  \frac{A^{(i_1,i_2)}}{C^{(i_1)}} \nonumber \\
\gamma_{01}^{(r)}  & = & \prod_{k=r}^{1}\Phi_{{\mathcal U}_k} (\gamma_0(i_1,i_2)) \nonumber \\
\gamma_{02}^{(r)}  & = & \prod_{k=r}^{1}\Phi_{{\mathcal U}_k} (\gamma_0(i_1,i_2)\times \gamma_0(i_1,p_2)) \nonumber \\
\gamma_{k_1+1}^{(r)}  & = & \prod_{k=r}^{k_1+1}\Phi_{{\mathcal U}_k} \lp \prod_{k=k_1}^{1}\Phi_{{\mathcal U}_k}\gamma_0(i_1,i_2)\times \prod_{k=k_1}^{1} \Phi_{{\mathcal U}_k}\gamma_0(p_1,p_2) \rp.
 \end{eqnarray}
Also, let
\begin{eqnarray}\label{eq:thm3eq41}
\phi^{(k_1,\p)} & = &
  -(1-t)\lp \m_{k_1}-\m_{k_1+1}\rp \mE_{G,{\mathcal U}_{r+1}} \langle \|\x^{(i_1)}\|_2\|\x^{(p_1)}\|_2(\y^{(p_2)})^T\y^{(i_2)} \rangle_{\gamma_{k_1+1}^{(1)}} \nonumber \\
& &   - t \q_{k_1}
\lp \m_{k_1}-\m_{k_1+1}\rp  \mE_{G,{\mathcal U}_{r+1}} \langle\|\x^{(i_1)}\|_2\|\x^{(p_1)}\|_2\|\y^{(i_2)}\|_2\|\y^{(p_2)}\|_2\rangle_{\gamma_{k_1+1}^{(1)}}\nonumber \\
\phi^{(k_1,\q)} & = &
   -(1-t)  \lp \m_{k_1}-\m_{k_1+1}\rp  \mE_{G,{\mathcal U}_{r+1}} \langle \|\y^{(i_2)}\|_2\|\y^{(p_2)}\|_2(\x^{(i_1)})^T\x^{(p_1)}\rangle_{\gamma_{k_1+1}^{(1)}} \nonumber \\
& &   - t \p_{k_1}
\lp \m_{k_1}-\m_{k_1+1}\rp  \mE_{G,{\mathcal U}_{r+1}} \langle\|\x^{(i_1)}\|_2\|\x^{(p_1)}\|_2\|\y^{(i_2)}\|_2\|\y^{(p_2)}\|_2\rangle_{\gamma_{k_1+1}^{(1)}}.
\end{eqnarray}
Then for $k_1\in\{1,2,\dots,r\}$
\begin{eqnarray}\label{eq:thm3eq42}
\frac{d\psi(\calX,\calY,\p,\q,\m,\beta,s,t)}{d\p_{k_1}}  & = &       \frac{\mbox{sign}(s)s\beta}{2\sqrt{n}} \phi^{(k_1,\p)}\nonumber \\
\frac{d\psi(\calX,\calY,\p,\q,\m,\beta,s,t)}{d\q_{k_1}}  & = &       \frac{\mbox{sign}(s)s\beta}{2\sqrt{n}} \phi^{(k_1,\q)}.
 \end{eqnarray}
 \end{theorem}
\begin{proof}
  Follows trivially through a combination of (\ref{eq:rthlev2genanal10e})-(\ref{eq:rthlev2genanal10f}), (\ref{eq:rthlev2genanal10e1})-(\ref{eq:rthlev2genanal10f1}), and (\ref{eq:rthlev2genanal27})-(\ref{eq:rthlev2genanal29}). \end{proof}

\section{Stationarization along the interpolating path}
\label{sec:statalongpath}

We start by recalling on the following fundamental result from \cite{Stojnicnflgscompyx23}.
\begin{theorem}
\label{thm:thm4}
Assume the setup of Theorem \ref{thm:thm3}. Let also
\begin{eqnarray}\label{eq:thm3eq5}
 \phi_{k_1}^{(r)} & = &
-s(\m_{k_1-1}-\m_{k_1}) \nonumber \\
&  & \times
\mE_{G,{\mathcal U}_{r+1}} \langle (\p_{k_1-1}\|\x^{(i_1)}\|_2\|\x^{(p_1)}\|_2 -(\x^{(p_1)})^T\x^{(i_1)})(\q_{k_1-1}\|\y^{(i_2)}\|_2\|\y^{(p_2)}\|_2 -(\y^{(p_2)})^T\y^{(i_2)})\rangle_{\gamma_{k_1}^{(r)}} \nonumber \\
 \phi_{01}^{(r)} & = & (1-\p_0)(1-\q_0)\mE_{G,{\mathcal U}_{r+1}}\langle \|\x^{(i_1)}\|_2^2\|\y^{(i_2)}\|_2^2\rangle_{\gamma_{01}^{(r)}} \nonumber\\
\phi_{02}^{(r)} & = & (s-1)(1-\p_0)\mE_{G,{\mathcal U}_{r+1}}\left\langle \|\x^{(i_1)}\|_2^2 \lp\q_0\|\y^{(i_2)}\|_2\|\y^{(p_2)}\|_2-(\y^{(p_2)})^T\y^{(i_2)}\rp\right\rangle_{\gamma_{02}^{(r)}}. \end{eqnarray}
Then
\begin{eqnarray}\label{eq:thm3eq6}
\frac{d\psi(\calX,\calY,\p,\q,\m,\beta,s,t)}{dt}  & = &       \frac{\mbox{sign}(s)\beta}{2\sqrt{n}} \lp  \lp\sum_{k_1=1}^{r+1} \phi_{k_1}^{(r)}\rp +\phi_{01}^{(r)}+\phi_{02}^{(r)}\rp.
 \end{eqnarray}
It particular, choosing $\p_0=\q_0=1$, one also has
\begin{eqnarray}\label{eq:rthlev2genanal43}
\frac{d\psi(\calX,\calY,\p,\q,\m,\beta,s,t)}{dt}  & = &       \frac{\mbox{sign}(s)\beta}{2\sqrt{n}} \sum_{k_1=1}^{r+1} \phi_{k_1}^{(r)} .
 \end{eqnarray}
 \end{theorem}
\begin{proof}
Presented in \cite{Stojnicnflgscompyx23}.
\end{proof}

Consider a new function $\psi_1$
\begin{eqnarray}\label{eq:saip1}
\psi_1(\calX,\calY,\p,\q,\m,\beta,s,t) & = & -\frac{\mbox{sign}(s) s \beta}{2\sqrt{n}} \nonumber \\
& & \times \sum_{k=1}^{r+1}\Bigg(\Bigg. \p_{k-1}\q_{k-1}\mE_{G,{\mathcal U}_{r+1}} \langle\|\x^{(i_1)}\|_2\|\x^{(p_1)}\|_2\|\y^{(i_2)}\|_2\|\y^{(p_2)}\|_2\rangle_{\gamma_{k}^{(r)}}\nonumber \\
& & -\p_{k}\q_{k}\mE_{G,{\mathcal U}_{r+1}} \langle\|\x^{(i_1)}\|_2\|\x^{(p_1)}\|_2\|\y^{(i_2)}\|_2\|\y^{(p_2)}\|_2\rangle_{\gamma_{k+1}^{(r)}}\Bigg.\Bigg)
\m_{k} \nonumber \\
& & +\psi(\calX,\calY,\p,\q,\m,\beta,s,t).
 \end{eqnarray}
 Then we also have
\begin{eqnarray}\label{eq:saip2}
\frac{d\psi_1(\calX,\calY,\p,\q,\m,\beta,s,t)}{d\p_{k_1}}
& = &
\frac{\mbox{sign}(s)s\beta}{2\sqrt{n}}\q_{k_1}\lp \m_{k_1} -\m_{k_1+1}\rp\mE_{G,{\mathcal U}_{r+1}} \langle\|\x^{(i_1)}\|_2\|\x^{(p_1)}\|_2\|\y^{(i_2)}\|_2\|\y^{(p_2)}\|_2\rangle_{\gamma_{k_1+1}^{(r)}} \nonumber \\
& & +\frac{\mbox{sign}(s)s\beta}{2\sqrt{n}} \phi^{(k_1,\p)} \nonumber \\
& = &
\frac{\mbox{sign}(s)s\beta}{2\sqrt{n}} \Bigg(\Bigg. \q_{k_1}\lp \m_{k_1} -\m_{k_1+1}\rp \mE_{G,{\mathcal U}_{r+1}} \langle\|\x^{(i_1)}\|_2\|\x^{(p_1)}\|_2\|\y^{(i_2)}\|_2\|\y^{(p_2)}\|_2\rangle_{\gamma_{k_1+1}^{(r)}}\nonumber \\
& &    -(1-t)\lp \m_{k_1}-\m_{k_1+1}\rp \mE_{G,{\mathcal U}_{r+1}} \langle \|\x^{(i_1)}\|_2\|\x^{(p_1)}\|_2(\y^{(p_2)})^T\y^{(i_2)} \rangle_{\gamma_{k_1+1}^{(1)}} \nonumber \\
& &   - t\q_{k_1}
\lp \m_{k_1}-\m_{k_1+1}\rp  \mE_{G,{\mathcal U}_{r+1}} \langle\|\x^{(i_1)}\|_2\|\x^{(p_1)}\|_2\|\y^{(i_2)}\|_2\|\y^{(p_2)}\|_2\rangle_{\gamma_{k_1+1}^{(r)}} \Bigg.\Bigg)\nonumber \\
& = &
(1-t)\lp \m_{k_1}-\m_{k_1+1}\rp\frac{\mbox{sign}(s)s\beta}{2\sqrt{n}} \nonumber \\
& & \times \Bigg(\Bigg.   \mE_{G,{\mathcal U}_{r+1}} \langle\|\x^{(i_1)}\|_2\|\x^{(p_1)}\|_2 \lp \q_{k_1} \|\y^{(i_2)}\|_2\|\y^{(p_2)}\|_2 -(\y^{(p_2)})^T\y^{(i_2)}\rp\rangle_{\gamma_{k_1+1}^{(r)}} \Bigg.\Bigg),\nonumber \\
 \end{eqnarray}
and
\begin{eqnarray}\label{eq:saip3}
\frac{d\psi_1(\calX,\calY,\p,\q,\m,\beta,s,t)}{d\q_{k_1}}
& = &
\frac{\mbox{sign}(s)s\beta}{2\sqrt{n}}\p_{k_1}\lp \m_{k_1} -\m_{k_1+1}\rp\mE_{G,{\mathcal U}_{r+1}} \langle\|\x^{(i_1)}\|_2\|\x^{(p_1)}\|_2\|\y^{(i_2)}\|_2\|\y^{(p_2)}\|_2\rangle_{\gamma_{k_1+1}^{(r)}} \nonumber \\
& & +\frac{\mbox{sign}(s)s\beta}{2\sqrt{n}} \phi^{(k_1,\q)} \nonumber \\
& = &
\frac{\mbox{sign}(s)s\beta}{2\sqrt{n}} \Bigg(\Bigg. \p_{k_1}\lp \m_{k_1} -\m_{k_1+1}\rp \mE_{G,{\mathcal U}_{r+1}} \langle\|\x^{(i_1)}\|_2\|\x^{(p_1)}\|_2\|\y^{(i_2)}\|_2\|\y^{(p_2)}\|_2\rangle_{\gamma_{k_1+1}^{(r)}}\nonumber \\
& &    -(1-t)\lp \m_{k_1}-\m_{k_1+1}\rp \mE_{G,{\mathcal U}_{r+1}} \langle \|\y^{(i_2)}\|_2\|\y^{(p_2)}\|_2(\x^{(p_1)})^T\x^{(i_1)} \rangle_{\gamma_{k_1+1}^{(1)}} \nonumber \\
& &   - t\p_{k_1}
\lp \m_{k_1}-\m_{k_1+1}\rp  \mE_{G,{\mathcal U}_{r+1}} \langle\|\x^{(i_1)}\|_2\|\x^{(p_1)}\|_2\|\y^{(i_2)}\|_2\|\y^{(p_2)}\|_2\rangle_{\gamma_{k_1+1}^{(r)}} \Bigg.\Bigg)\nonumber \\
& = &
(1-t)\lp \m_{k_1}-\m_{k_1+1}\rp\frac{\mbox{sign}(s)s\beta}{2\sqrt{n}} \nonumber \\
& & \times \Bigg(\Bigg.   \mE_{G,{\mathcal U}_{r+1}} \langle\|\y^{(i_2)}\|_2\|\y^{(p_2)}\|_2
\lp \p_{k_1} \|\x^{(i_1)}\|_2\|\x^{(p_1)}\|_2 -(\x^{(p_1)})^T\x^{(i_1)}\rp\rangle_{\gamma_{k_1+1}^{(r)}} \Bigg.\Bigg).\nonumber \\
 \end{eqnarray}
The above, of course, holds for any $t$. Below we focus on vectors $\p(t)$, $\q(t)$, and $\m(t)$ as functions of $t$ and assume that they satisfy the following system of equations
\begin{eqnarray}\label{eq:saip4}
\frac{d\psi_1(\calX,\calY,\p,\q,\m,\beta,s,t)}{d\p_{k_1}}
& = &
(1-t)\lp \m_{k_1}-\m_{k_1+1}\rp\frac{\mbox{sign}(s)s\beta}{2\sqrt{n}} \nonumber \\
& & \times \Bigg(\Bigg.   \mE_{G,{\mathcal U}_{r+1}} \langle\|\x^{(i_1)}\|_2\|\x^{(p_1)}\|_2 \lp \q_{k_1} \|\y^{(i_2)}\|_2\|\y^{(p_2)}\|_2 -(\y^{(p_2)})^T\y^{(i_2)}\rp\rangle_{\gamma_{k_1+1}^{(r)}} \Bigg.\Bigg)\nonumber \\
& = & 0, \nonumber \\
\frac{d\psi_1(\calX,\calY,\p,\q,\m,\beta,s,t)}{d\q_{k_1}}
& = &
(1-t)\lp \m_{k_1}-\m_{k_1+1}\rp\frac{\mbox{sign}(s)s\beta}{2\sqrt{n}} \nonumber \\
& & \times \Bigg(\Bigg.   \mE_{G,{\mathcal U}_{r+1}} \langle\|\y^{(i_2)}\|_2\|\y^{(p_2)}\|_2
\lp \p_{k_1} \|\x^{(i_1)}\|_2\|\x^{(p_1)}\|_2 -(\x^{(p_1)})^T\x^{(i_1)}\rp\rangle_{\gamma_{k_1+1}^{(r)}} \Bigg.\Bigg) \nonumber \\
 & = & 0, \nonumber \\
\frac{d\psi_1(\calX,\calY,\p,\q,\m,\beta,s,t)}{d\m_{k_1}}
 & = & 0,
  \end{eqnarray}
where $k_1\in\{1,2,\dots,r\}$ and $\p_0(t)=\q_0(t)=\m_0(t)=1$. We then have the following theorem.

\begin{theorem}
\label{thm:thm5}
Consider the following completely stationirized  fully lifted random duality theory frame (\textbf{\emph{complete sfl RDT frame}}). Assume the existence of the solutions $\bar{\p}(t)$, $\bar{\q}(t)$, and $\bar{\m}(t)$ of system (\ref{eq:saip4}) and that $\bar{\p}(t)$ and $\bar{\q}(t)$ are not only the (appropriately scaled/normalized) expected values (in the sense of (\ref{eq:saip4})) of the overlaps $(\x^{(p_1)})^T\x^{(i_1)}$ and $(\y^{(p_2)})^T\y^{(i_2)}$ but also their concentrating points (or such that $\phi_{k_1+1}^{(r)}=0$). For $\bar{\p}_0(t)=\bar{\q}_0(t)=1$,
$\bar{\p}_{r+1}(t)=\bar{\q}_{r+1}(t)=\bar{\m}_{r+1}(t)=0$, and $\m_1(t)\rightarrow\m_0(t)=1$ one then has that
$\frac{d\psi_1(\calX,\calY,\bar{\p}(t),\bar{\q}(t),\bar{\m}(t),\beta,s,t)}{dt}   =   0$ and
\begin{eqnarray}\label{eq:thm5eq1}
\lim_{n\rightarrow\infty}\psi_1(\calX,\calY,\bar{\p}(t),\bar{\q}(t),\bar{\m}(t),\beta,s,t)
& = &
\lim_{n\rightarrow\infty}\psi_1(\calX,\calY,\bar{\p}(0),\bar{\q}(0),\bar{\m}(0),\beta,s,0) \nonumber \\
& = & \lim_{n\rightarrow\infty} \psi_1(\calX,\calY,\bar{\p}(1),\bar{\q}(1),\bar{\m}(1),\beta,s,1),
\end{eqnarray}
where $\psi_1(\cdot)$ and $\psi(\cdot)$ are as in (\ref{eq:saip1}) and  (\ref{eq:thm3eq1}), respectively.
\end{theorem}

\begin{proof}
  Recalling on (\ref{eq:thm3eq5}), we have
\begin{eqnarray}\label{eq:saip4a}
 \phi_{k_1+1}^{(r)}
 & = &
-s(\bar{\m}_{k_1}(t)-\bar{\m}_{k_1+1}(t)) \nonumber \\
&  & \times
\mE_{G,{\mathcal U}_{r+1}} \langle (\bar{\p}_{k_1}(t)\|\x^{(i_1)}\|_2\|\x^{(p_1)}\|_2 -(\x^{(p_1)})^T\x^{(i_1)})(\bar{\q}_{k_1}(t)\|\y^{(i_2)}\|_2\|\y^{(p_2)}\|_2 -(\y^{(p_2)})^T\y^{(i_2)})\rangle_{\gamma_{k_1+1}^{(r)}} \nonumber \\
& = &
-s(\bar{\m}_{k_1}(t)-\bar{\m}_{k_1+1}(t)) \nonumber \\
&  & \times
\mE_{G,{\mathcal U}_{r+1}} \langle (\bar{\p}_{k_1}(t)\|\x^{(i_1)}\|_2\|\x^{(p_1)}\|_2 -(\x^{(p_1)})^T\x^{(i_1)})\bar{\q}_{k_1}(t)\|\y^{(i_2)}\|_2\|\y^{(p_2)}\|_2 \rangle_{\gamma_{k_1+1}^{(r)}} \nonumber \\
&  &
+s(\bar{\m}_{k_1}(t)-\bar{\m}_{k_1+1}(t)) \nonumber \\
&  & \times
\mE_{G,{\mathcal U}_{r+1}} \langle (\bar{\p}_{k_1}(t)\|\x^{(i_1)}\|_2\|\x^{(p_1)}\|_2 -(\x^{(p_1)})^T\x^{(i_1)})(\y^{(p_2)})^T\y^{(i_2)}\rangle_{\gamma_{k_1+1}^{(r)}}.
\end{eqnarray}
Utilizing the second part of (\ref{eq:saip4}), we easily have
\begin{eqnarray}\label{eq:saip5}
 \phi_{k_1+1}^{(r)}
 & = &
  s(\bar{\m}_{k_1}(t)-\bar{\m}_{k_1+1}(t)) \nonumber \\
&  & \times
\mE_{G,{\mathcal U}_{r+1}} \langle (\bar{\p}_{k_1}(t)\|\x^{(i_1)}\|_2\|\x^{(p_1)}\|_2 -(\x^{(p_1)})^T\x^{(i_1)})(\y^{(p_2)})^T\y^{(i_2)}\rangle_{\gamma_{k_1+1}^{(r)}} \nonumber \\
 & = &
  s(\bar{\m}_{k_1}(t)-\bar{\m}_{k_1+1}(t)) \nonumber \\
&  & \times
\mE_{G,{\mathcal U}_{r+1}} \langle ( -(\x^{(p_1)})^T\x^{(i_1)})(\y^{(p_2)})^T\y^{(i_2)}\rangle_{\gamma_{k_1+1}^{(r)}} \nonumber \\
 &  &
  +s(\bar{\m}_{k_1}(t)-\bar{\m}_{k_1+1}(t)) \nonumber \\
&  & \times
\mE_{G,{\mathcal U}_{r+1}} \langle (\bar{\p}_{k_1}(t)\|\x^{(i_1)}\|_2\|\x^{(p_1)}\|_2)(\y^{(p_2)})^T\y^{(i_2)}\rangle_{\gamma_{k_1+1}^{(r)}} \nonumber \\
 & = &
  s(\bar{\m}_{k_1}(t)-\bar{\m}_{k_1+1}(t)) \nonumber \\
&  & \times
\mE_{G,{\mathcal U}_{r+1}} \langle ( -(\x^{(p_1)})^T\x^{(i_1)})(\y^{(p_2)})^T\y^{(i_2)}\rangle_{\gamma_{k_1+1}^{(r)}} \nonumber \\
 &  &
  +s(\bar{\m}_{k_1}(t)-\bar{\m}_{k_1+1}(t)) \nonumber \\
&  & \times
\mE_{G,{\mathcal U}_{r+1}} \langle \bar{\p}_{k_1}(t)\|\x^{(i_1)}\|_2\|\x^{(p_1)}\|_2
\bar{\q}_{k_1}(t) \|\y^{(p_2)}\|_2\|\y^{(i_2)}\|_2\rangle_{\gamma_{k_1+1}^{(r)}} \nonumber \\
 &  &
  -s(\bar{\m}_{k_1}(t)-\bar{\m}_{k_1+1}(t)) \nonumber \\
&  & \times
\mE_{G,{\mathcal U}_{r+1}} \langle \bar{\p}_{k_1}(t)\|\x^{(i_1)}\|_2\|\x^{(p_1)}\|_2
\bar{\q}_{k_1}(t) \|\y^{(p_2)}\|_2\|\y^{(i_2)}\|_2\rangle_{\gamma_{k_1+1}^{(r)}} \nonumber \\
 &  &
  +s(\bar{\m}_{k_1}(t)-\bar{\m}_{k_1+1}(t)) \nonumber \\
&  & \times
\mE_{G,{\mathcal U}_{r+1}} \langle (\bar{\p}_{k_1}(t)\|\x^{(i_1)}\|_2\|\x^{(p_1)}\|_2)(\y^{(p_2)})^T\y^{(i_2)}\rangle_{\gamma_{k_1+1}^{(r)}} \nonumber \\
 & = &
  s(\bar{\m}_{k_1}(t)-\bar{\m}_{k_1+1}(t)) \nonumber \\
&  & \times
\mE_{G,{\mathcal U}_{r+1}} \langle
\bar{\p}_{k_1}(t)\|\x^{(i_1)}\|_2\|\x^{(p_1)}\|_2
\bar{\q}_{k_1}(t) \|\y^{(p_2)}\|_2\|\y^{(i_2)}\|_2
 -(\x^{(p_1)})^T\x^{(i_1)}(\y^{(p_2)})^T\y^{(i_2)}\rangle_{\gamma_{k_1+1}^{(r)}} \nonumber \\
  &  &
  -s(\bar{\m}_{k_1}(t)-\bar{\m}_{k_1+1}(t)) \nonumber \\
&  & \times
\mE_{G,{\mathcal U}_{r+1}} \langle \bar{\p}_{k_1}(t)\|\x^{(i_1)}\|_2\|\x^{(p_1)}\|_2
\lp\bar{\q}_{k_1}(t) \|\y^{(p_2)}\|_2\|\y^{(i_2)}\|_2
- (\y^{(p_2)})^T\y^{(i_2)}  \rp  \rangle_{\gamma_{k_1+1}^{(r)}}.
 \end{eqnarray}
Utilizing the first part of (\ref{eq:saip4}), we now easily also have
\begin{eqnarray}\label{eq:saip6}
 \phi_{k_1+1}^{(r)}
  & = &
  s(\bar{\m}_{k_1}(t)-\bar{\m}_{k_1+1}(t)) \nonumber \\
&  & \times
\mE_{G,{\mathcal U}_{r+1}} \langle
\bar{\p}_{k_1}(t)\|\x^{(i_1)}\|_2\|\x^{(p_1)}\|_2
\bar{\q}_{k_1}(t) \|\y^{(p_2)}\|_2\|\y^{(i_2)}\|_2
 -(\x^{(p_1)})^T\x^{(i_1)}(\y^{(p_2)})^T\y^{(i_2)}\rangle_{\gamma_{k_1+1}^{(r)}}.   \nonumber \\
 \end{eqnarray}
By the Cauchy-Schwartz inequality one has $|\mbox{Cov}(a,b)|\leq \sqrt{\mbox{Var}(a)}\sqrt{\mbox{Var}(b)}$. Combining this with the scaled overlaps concentrations gives for $n\rightarrow\infty$
\begin{eqnarray}\label{eq:saip6a00}
a_{\phi} & \triangleq &\mE_{G,{\mathcal U}_{r+1}} \bigg | \bigg. \langle (\bar{\p}_{k_1}(t)\|\x^{(i_1)}\|_2\|\x^{(p_1)}\|_2 -(\x^{(p_1)})^T\x^{(i_1)})(\bar{\q}_{k_1}(t)\|\y^{(i_2)}\|_2\|\y^{(p_2)}\|_2 -(\y^{(p_2)})^T\y^{(i_2)})\rangle_{\gamma_{k_1+1}^{(r)}} \nonumber \\
& &
-  \langle \bar{\p}_{k_1}(t)\|\x^{(i_1)}\|_2\|\x^{(p_1)}\|_2 -(\x^{(p_1)})^T\x^{(i_1)} \rangle_{\gamma_{k_1+1}^{(r)}} \langle (\bar{\q}_{k_1}(t)\|\y^{(i_2)}\|_2\|\y^{(p_2)}\|_2 -(\y^{(p_2)})^T\y^{(i_2)} \rangle_{\gamma_{k_1+1}^{(r)}} \bigg. \bigg|^2\nonumber \\
& \leq &\mE_{G,{\mathcal U}_{r+1}} \Bigg | \Bigg.
\Bigg( \Bigg.  \langle (\bar{\p}_{k_1}(t)\|\x^{(i_1)}\|_2\|\x^{(p_1)}\|_2 -(\x^{(p_1)})^T\x^{(i_1)})^2\rangle_{\gamma_{k_1+1}^{(r)}} \nonumber \\
& &
- \lp\langle (\bar{\p}_{k_1}(t)\|\x^{(i_1)}\|_2\|\x^{(p_1)}\|_2 -(\x^{(p_1)})^T\x^{(i_1)})\rangle_{\gamma_{k_1+1}^{(r)}}\rp^2 \Bigg.\Bigg)\nonumber \\
& & \times
\Bigg( \Bigg. \bigg. \langle (\bar{\q}_{k_1}(t)\|\y^{(i_2)}\|_2\|\y^{(p_2)}\|_2 -(\y^{(p_2)})^T\y^{(i_2)})^2\rangle_{\gamma_{k_1+1}^{(r)}} \nonumber \\
& &
- \lp\langle (\bar{\q}_{k_1}(t)\|\y^{(i_2)}\|_2\|\y^{(p_2)}\|_2 -(\y^{(p_2)})^T\y^{(i_2)})\rangle_{\gamma_{k_1+1}^{(r)}}\rp^2 \Bigg.\Bigg)
 \Bigg. \Bigg|\rightarrow 0.
 \end{eqnarray}
Since the scaled overlaps concentrations also ensure
\begin{eqnarray}\label{eq:saip6a01}
 \mE_{G,{\mathcal U}_{r+1}}
 \langle \bar{\p}_{k_1}(t)\|\x^{(i_1)}\|_2\|\x^{(p_1)}\|_2 -(\x^{(p_1)})^T\x^{(i_1)} \rangle_{\gamma_{k_1+1}^{(r)}} \langle (\bar{\q}_{k_1}(t)\|\y^{(i_2)}\|_2\|\y^{(p_2)}\|_2 -(\y^{(p_2)})^T\y^{(i_2)} \rangle_{\gamma_{k_1+1}^{(r)}} \rightarrow 0,
 \end{eqnarray}
one then has that $a_{\phi}\rightarrow 0$ also implies
\begin{eqnarray}\label{eq:saip6a02}
 \mE_{G,{\mathcal U}_{r+1}}   \langle (\bar{\p}_{k_1}(t)\|\x^{(i_1)}\|_2\|\x^{(p_1)}\|_2 -(\x^{(p_1)})^T\x^{(i_1)})(\bar{\q}_{k_1}(t)\|\y^{(i_2)}\|_2\|\y^{(p_2)}\|_2 -(\y^{(p_2)})^T\y^{(i_2)})\rangle_{\gamma_{k_1+1}^{(r)}} \rightarrow 0.
 \end{eqnarray}
Combining (\ref{eq:saip6a02}) with (\ref{eq:saip6}) produces a, possibly more clearly spelled out, overlaps-products (correlations) concentrations,
$\phi_{k_1+1}^{(r)} \rightarrow 0$, i.e.,
\begin{eqnarray}\label{eq:saip6a}
 \phi_{k_1+1}^{(r)}
  & = &
  s(\bar{\m}_{k_1}(t)-\bar{\m}_{k_1+1}(t)) \nonumber \\
&  & \times
\mE_{G,{\mathcal U}_{r+1}} \langle
\bar{\p}_{k_1}(t)\|\x^{(i_1)}\|_2\|\x^{(p_1)}\|_2
\bar{\q}_{k_1}(t) \|\y^{(p_2)}\|_2\|\y^{(i_2)}\|_2
 -(\x^{(p_1)})^T\x^{(i_1)}(\y^{(p_2)})^T\y^{(i_2)}\rangle_{\gamma_{k_1+1}^{(r)}}   \nonumber \\
& \rightarrow & 0,
\end{eqnarray}
and based on (\ref{eq:rthlev2genanal43})
\begin{eqnarray}\label{eq:saip7}
\frac{d\psi_1(\calX,\calY,\bar{\p}(t),\bar{\q}(t),\bar{\m}(t),\beta,s,t)}{dt}
& = &
\frac{\partial\psi(_1\calX,\calY,\bar{\p}(t),\bar{\q}(t),\bar{\m}(t),\beta,s,t)}{\partial t}\frac{dt}{dt} \nonumber \\
& &
+\frac{\partial \psi_1(\calX,\calY,\bar{\p}(t),\bar{\q}(t),\bar{\m}(t),\beta,s,t)}{\partial \bar{\p}(t)}\frac{d\bar{\p}(t)}{dt} \nonumber \\
& &
+\frac{\partial \psi_1(\calX,\calY,\bar{\p}(t),\bar{\q}(t),\bar{\m}(t),\beta,s,t)}{\partial \bar{\q}(t)}\frac{d\bar{\q}(t)}{dt} \nonumber \\
& &
+\frac{\partial \psi_1(\calX,\calY,\bar{\p}(t),\bar{\q}(t),\bar{\m}(t),\beta,s,t)}{\partial \bar{\m}(t)}\frac{d\bar{\m}(t)}{dt} \nonumber \\
& = &
\frac{\partial\psi_1(\calX,\calY,\bar{\p}(t),\bar{\q}(t),\bar{\m}(t),\beta,s,t)}{\partial t}
 \nonumber \\
& = &
\frac{\partial\psi(\calX,\calY,\bar{\p}(t),\bar{\q}(t),\bar{\m}(t),\beta,s,t)}{\partial t}
 \nonumber \\
& = &       \frac{\mbox{sign}(s)\beta}{2\sqrt{n}} \sum_{k_1=1}^{r+1} \phi_{k_1}^{(r)}\rightarrow 0,
 \end{eqnarray}
where the second equality follows by the choice of $\bar{\p}(t)$, $\bar{\q}(t)$, and $\bar{\m}(t)$ (i.e., by them being the stationary points of $\psi_1(\cdot)$).
\end{proof}

We also immediately have the following, practically very relevant, corollary
\begin{corollary}
\label{thm:thm6}
Assume the setup of Theorem \ref{thm:thm5}. Let ${\mathcal X}$ and ${\mathcal Y}$ be such that $\|\x\|_2=\|\y\|_2=1$ (or, alternatively, such that their concentrating values in the sfl RDT frame sense are $1$). Then
 \begin{eqnarray}\label{eq:thm6eq0}
\lim_{n\rightarrow\infty}\psi_1(\calX,\calY,\bar{\p}(t),\bar{\q}(t),\bar{\m}(t),\beta,s,t)
& = &
\lim_{n\rightarrow\infty}\psi_1(\calX,\calY,\bar{\p}(0),\bar{\q}(0),\bar{\m}(0),\beta,s,0) \nonumber \\
& = & \lim_{n\rightarrow\infty} \psi_1(\calX,\calY,\bar{\p}(1),\bar{\q}(1),\bar{\m}(1),\beta,s,1),
\end{eqnarray}
 and
\begin{eqnarray}\label{eq:thm6eq1}
\psi_1(\calX,\calY,\bar{\p}(0),\bar{\q}(0),\bar{\m}(0),\beta,s,0) & = & -\frac{\mbox{sign}(s) s \beta}{2\sqrt{n}} \sum_{k=1}^{r+1}\Bigg(\Bigg. \bar{\p}_{k-1}(0)\bar{\q}_{k-1}(0)  -\bar{\p}_{k}(0)\bar{\q}_{k}(0)   \Bigg.\Bigg)
\bar{\m}_k(0) \nonumber \\
& &  +\psi(\calX,\calY,\bar{\p}(0),\bar{\q}(0),\bar{\m}(0),\beta,s,0) \nonumber \\
\psi_1(\calX,\calY,\bar{\p}(1),\bar{\q}(1),\bar{\m}(1),\beta,s,1) & = & -\frac{\mbox{sign}(s) s \beta}{2\sqrt{n}} \sum_{k=1}^{r+1}\Bigg(\Bigg. \bar{\p}_{k-1}(1)\bar{\q}_{k-1}(1)  -\bar{\p}_{k}(1)\bar{\q}_{k}(1)   \Bigg.\Bigg)
\bar{\m}_k(1)  \nonumber \\
& &  +\psi(\calX,\calY,\bar{\p}(1),\bar{\q}(1),\bar{\m}(1),\beta,s,1).
 \end{eqnarray}
 Moreover, let
 \begin{equation}\label{eq:thm6eq2}
\psi_S(\calX,\calY,\p,\q,\m,\beta,s,t)  =  \mE_{G,{\mathcal U}_{r+1}} \frac{1}{\beta|s|\sqrt{n}\m_r} \log
\lp \mE_{{\mathcal U}_{r}} \lp \dots \lp \mE_{{\mathcal U}_2}\lp\lp\mE_{{\mathcal U}_1} \lp Z_S^{\m_1}\rp\rp^{\frac{\m_2}{\m_1}}\rp\rp^{\frac{\m_3}{\m_2}} \dots \rp^{\frac{\m_{r}}{\m_{r-1}}}\rp,
\end{equation}
where,  analogously to (\ref{eq:thm3eq1}) and (\ref{eq:thm3eq2}),
\begin{equation}\label{eq:thm6eq3}
Z_S  \triangleq  \sum_{i_1=1}^{l}\lp\sum_{i_2=1}^{l}e^{\beta D_{0,S}^{(i_1,i_2)}} \rp^{s}, \nonumber \\
\end{equation}
with
\begin{equation}\label{eq:thm6eq4}
 D_{0,S}^{(i_1,i_2)}  \triangleq  \sqrt{t}(\y^{(i_2)})^T
 G\x^{(i_1)}+\sqrt{1-t}\|\x^{(i_2)}\|_2 (\y^{(i_2)})^T\lp\sum_{k=1}^{r+1}b_k\u^{(2,k)}\rp  +\sqrt{1-t}\|\y^{(i_2)}\|_2\lp\sum_{k=1}^{r+1}c_k\h^{(k)}\rp^T\x^{(i)}.
 \end{equation}
Then
\begin{eqnarray}\label{eq:thm6eq5}
\lim_{n\rightarrow\infty} \psi_S(\calX,\calY,\bar{\p}(1),\bar{\q}(1),\bar{\m}(1),\beta,s,1) & = &
 -\lim_{n\rightarrow\infty}\frac{\mbox{sign}(s) s \beta}{2\sqrt{n}} \sum_{k=1}^{r+1}\Bigg(\Bigg. \bar{\p}_{k-1}(0)\bar{\q}_{k-1}(0)  -\bar{\p}_{k}(0)\bar{\q}_{k}(0)   \Bigg.\Bigg)
\bar{\m}_k(t)
 \nonumber \\
& &  +\lim_{n\rightarrow\infty} \psi_S(\calX,\calY,\bar{\p}(0),\bar{\q}(0),\bar{\m}(0),\beta,s,0). \nonumber \\
 \end{eqnarray}
\end{corollary}
\begin{proof}
After integrating out $u^{(2,k)}$ in $\psi(\cdot)$, one observes
\begin{eqnarray}\label{eq:thm6eq6}
 \psi_S(\calX,\calY,\bar{\p}(1),\bar{\q}(1),\bar{\m}(1),\beta,s,1) & = & -\frac{\mbox{sign}(s) s \beta}{2\sqrt{n}} \sum_{k=1}^{r+1}\Bigg(\Bigg. \bar{\p}_{k-1}(1)\bar{\q}_{k-1}(1)  -\bar{\p}_{k}(1)\bar{\q}_{k}(1)   \Bigg.\Bigg)
\bar{\m}_k(1)
\nonumber \\
& &  +\psi(\calX,\calY,\bar{\p}(1),\bar{\q}(1),\bar{\m}(1),\beta,s,1). \nonumber \\
& = &  \psi_1(\calX,\calY,\bar{\p}(1),\bar{\q}(1),\bar{\m}(1),\beta,s,1).
 \end{eqnarray}
Also,
 \begin{eqnarray}\label{eq:thm6eq7}
  \psi_1(\calX,\calY,\bar{\p}(0),\bar{\q}(0),\bar{\m}(0),\beta,s,0) & = & -\frac{\mbox{sign}(s) s \beta}{2\sqrt{n}} \sum_{k=1}^{r+1}\Bigg(\Bigg. \bar{\p}_{k-1}(0)\bar{\q}_{k-1}(0)  -\bar{\p}_{k}(0)\bar{\q}_{k}(0)   \Bigg.\Bigg)
\bar{\m}_k(0) \nonumber \\
& &  + \psi(\calX,\calY,\bar{\p}(0),\bar{\q}(0),\bar{\m}(0),\beta,s,0) \nonumber \\
 & = &  -\frac{\mbox{sign}(s) s \beta}{2\sqrt{n}} \sum_{k=1}^{r+1}\Bigg(\Bigg. \bar{\p}_{k-1}(0)\bar{\q}_{k-1}(0)  -\bar{\p}_{k}(0)\bar{\q}_{k}(0)   \Bigg.\Bigg)
\bar{\m}_k(0) \nonumber \\
& &  +\psi_S(\calX,\calY,\bar{\p}(0),\bar{\q}(0),\bar{\m}(0),\beta,s,0). \nonumber \\
 \end{eqnarray}
Since   $\lim_{n\rightarrow\infty}\psi_1(\calX,\calY,\bar{\p}(0),\bar{\q}(0),\bar{\m}(0),\beta,s,0)
=\lim_{n\rightarrow\infty}\psi_1(\calX,\calY,\bar{\p}(1),\bar{\q}(1),\bar{\m}(1),\beta,s,1)$, is implied by (\ref{eq:thm5eq1}) and (\ref{eq:thm6eq0}), one immediately has
\begin{eqnarray}\label{eq:thm6eq8}
\lim_{n\rightarrow\infty} \psi_S(\calX,\calY,\bar{\p}(1),\bar{\q}(1),\bar{\m}(1),\beta,s,1) & = & -\lim_{n\rightarrow\infty}\frac{\mbox{sign}(s) s \beta}{2\sqrt{n}} \sum_{k=1}^{r+1}\Bigg(\Bigg. \bar{\p}_{k-1}(0)\bar{\q}_{k-1}(0)  -\bar{\p}_{k}(0)\bar{\q}_{k}(0)   \Bigg.\Bigg)
\bar{\m}_k(0) \nonumber \\
& &   +\lim_{n\rightarrow\infty}\psi_S(\calX,\calY,\bar{\p}(0),\bar{\q}(0),\bar{\m}(0),\beta,s,0), \nonumber \\
 \end{eqnarray}
which is exactly the same as (\ref{eq:thm6eq5}).
\end{proof}

The following corollary is then a trivial consequence of Theorem \ref{thm:thm5} and Corollary \ref{thm:thm6}.
\begin{corollary}
\label{thm:thm7}
Consider the following modulo $\m$  stationirized fully lifted random duality theory frame (\textbf{\emph{modulo-$\m$ sfl RDT frame}}). Assume the existence of $\m$ such that the solutions $\bar{\p}(t)$ and $\bar{\q}(t)$ (obtained for such an $\m$) to the first two sequences of equations of system (\ref{eq:saip4})  are not only the
 (appropriately scaled/normalized) expected values (in the sense of (\ref{eq:saip4})) of the overlaps $(\x^{(p_1)})^T\x^{(i_1)}$ and $(\y^{(p_2)})^T\y^{(i_2)}$, but also their concentrating points (or such that $\phi_{k_1+1}^{(r)}=0$). Moreover, let ${\mathcal X}$ and ${\mathcal Y}$ be as in Corollary \ref{thm:thm6}. For $\bar{\p}_0(t)=1$, $\bar{\q}_0(t)=1$, and $\m_1\rightarrow\m_0=1$ one then has that
$\frac{d\psi_1(\calX,\calY,\bar{\p}(t),\bar{\q}(t),\m,\beta,s,t)}{dt}   =   0$,
 \begin{eqnarray}\label{eq:thm7eq0}
\lim_{n\rightarrow\infty}\psi_1(\calX,\calY,\bar{\p}(t),\bar{\q}(t),\bar{\m}(t),\beta,s,t)
& = &
\lim_{n\rightarrow\infty}\psi_1(\calX,\calY,\bar{\p}(0),\bar{\q}(0),\bar{\m}(0),\beta,s,0) \nonumber \\
& = & \lim_{n\rightarrow\infty} \psi_1(\calX,\calY,\bar{\p}(1),\bar{\q}(1),\bar{\m}(1),\beta,s,1),
\end{eqnarray}
and
\begin{eqnarray}\label{eq:thm7eq1}
\lim_{n\rightarrow\infty} \psi_S(\calX,\calY,\bar{\p}(1),\bar{\q}(1),\m,\beta,s,1) & \geq  &  \lim_{n\rightarrow\infty}  \inf_{\m} \Bigg(\Bigg.
-\frac{\mbox{sign}(s) s \beta}{2\sqrt{n}} \sum_{k=1}^{r+1}\Bigg(\Bigg. \bar{\p}_{k-1}(0)\bar{\q}_{k-1}(0)  -\bar{\p}_{k}(0)\bar{\q}_{k}(0)   \Bigg.\Bigg)
\m_{k}  \nonumber \\
& &   +   \psi_S(\calX,\calY,\bar{\p}(0),\bar{\q}(0),\m,\beta,s,0)\Bigg.\Bigg). \nonumber \\
 \end{eqnarray}
\end{corollary}
\begin{proof}
   Follows by trivially adjusting the proofs of Theorem \ref{thm:thm5} and Corollaries \ref{thm:thm6} and \ref{thm:thm7}.
\end{proof}


Analogous  modulo-$(\m,\p)$ or modulo-$(\m,\q)$ concepts can be  considered as well, but we skip such trivialities. Instead, we point out a couple of useful things that should be kept in mind: 1) All that we presented above was in a discrete (countable) domain. In other words, we worked with countable sequences of (potentially infinite) length $r$. Everything holds if, instead, rephrased in the corresponding continuous domain as well. In fact, one can then view $\m(t)$ as a function of $\p(t)$ and $\q(t)$, $\m(\p(t),\q(t))$, and work with continuity of $\m(t)$ to ensure the above needed existences; 2) To make writing and notation easier to handle, we assumed $\p_{r+1}(t)=\q_{r+1}(t)=0$. Such an assumption is not required though. Instead, with a little bit of additional work and finer adjustment, one can rewrite pretty much everything with, say, $\p_{r+1}(t)=\q_{r+1}(t)=-1$. This is particularly relevant as it allows to guarantee a switch of sign throughout the range of admissible $\p(t)$'s and $\q(t)$'s in the concentrating expression underlying (\ref{eq:saip4}) and ultimately $\phi_{k_1+1}^{(r)}$ in (\ref{eq:saip6}) and (\ref{eq:saip6a}). Together with the continuity of $\p(t)$, $\q(t)$, and $\m(\p(t),\q(t))$ and the fact that the scaled $\x$'s and $\y$'s overlaps belong to the interval $[-1,1]$, this establishes a useful combination to help ensure that the equalities in the given system of equations (\ref{eq:saip4}) can indeed be achieved; 3) In particular, assuming that one can determine \emph{unique} fully continuous (as extensions of the corresponding infinitely ($r\rightarrow\infty$) long discrete sequences) $\p(0)$, $\q(0)$, and $\m(\p(0),\q(0))$ such that (\ref{eq:saip4})  and (\ref{eq:saip6a}) are satisfied for $t=0$ (i.e., for the the fully decoupled model), the continuities of $\p(t)$ and $\q(t)$ over $t$ and the continuities of $\m(\p(t),\q(t))$ over both $t$ and $\p(t)$ and $\q(t)$ are rather valuable. On the other hand, the nonexistence of fully continuous such $\p(0)$, $\q(0)$, and $\m(\p(0),\q(0))$ together with the continuity of $\p(t)$, $\q(t)$, and $\m(t)$ over $t$ is valuable as well; 4) The same properties related to the continuity of ($\p(t)$ and $\q(t)$) over $t$ for modulo-$\m$ frame (where $\p(0)$, $\q(0)$, and $\m(\p(0),\q(0))$ are such that the first two sequences/functions of equations of (\ref{eq:saip4})  and (\ref{eq:saip6a}) are satisfied for $t=0$) are valuable as well; 5) When $t=0$ full decoupling happens on two levels: (i) the components of both $\x$ and $\y$ are completely decoupled; and (ii) the $\gamma$ measures decouple over $\x$ and $\y$. The $\gamma$ decoupling is already sufficient to ensure that, for $\bar{\p}(0)$ and $\bar{\q}(0)$ that satisfy (\ref{eq:saip4}), one also has that $\phi_{k_1+1}^{(r)}=0$ for $t=0$ as well. On the other hand, assuming $\lim_{n\rightarrow\infty}\frac{r}{n}=0$, the decoupling over the components of $\x$ and $\y$ (together with the $\gamma$ decoupling) trivially enures that the scaled overlaps concentrate for $t=0$ as well.

\subsection{Convexity}
\label{sec:convexity}

One should note that the expression inside $\phi_{k_1+1}^{(r)}$ does concentrate for any  pairs $(i_2,p_2)$ or $(i_1,p_1)$. When one runs over all such pairs (say $(i_2,p_2)$), things are a little bit different. However, for negative $s$ and convex ${\mathcal X}$ and ${\mathcal Y}$, one has that $(C^{(i_1)})^s$ is log-concave. Namely, $\log\lp (C^{(i_1)})^s\rp=s\log\lp C^{(i_1)}\rp=-|s|\log\lp C^{(i_1)}\rp$ which is concave since the $\log (\cdot)$ of a sum of exponentials with convex (linear) exponents is convex. The Brascamp-Lieb inequality (see, e.g., \cite{BraLieb76} and \cite{Prek71,Prek73,Lein71,HadOhm56,Lus35} as well) then ensures the concentration under product measure $\frac{(C^{(i_1)})^s}{Z}$ and the whole above machinery works in full power. This, of course, is also a well known consequence of the basics of random duality theory, where the above theorems and corollaries hold (with components of all $\p$, $\q$, $\m$ in $\beta\rightarrow\infty $ regime  going to either zero or one) for problems where the strong deterministic duality holds (and therefore the convex ones as well) (for more details see, e.g., \cite{StojnicRegRndDlt10,StojnicUpper10,StojnicGorEx10,StojnicGenLasso10}).

\section{A few bilinear model well known examples}
\label{sec:examples}

We here briefly recall on several well known practical examples that can fit into the model considered above (see also, e.g., \cite{Stojnicnflgscompyx23}).
As is clear from the above presentation, the key object of interest is the following function from Theorem \ref{thm:thm3}
\begin{equation}\label{eq:exampleseq1}
\psi(\calX,\calY,\p,\q,\m,\beta,s,t)  =  \mE_{G,{\mathcal U}_{r+1}} \frac{1}{\beta|s|\sqrt{n}\m_r} \log
\lp \mE_{{\mathcal U}_{r}} \lp \dots \lp \mE_{{\mathcal U}_2}\lp\lp\mE_{{\mathcal U}_1} \lp Z^{\m_1}\rp\rp^{\frac{\m_2}{\m_1}}\rp\rp^{\frac{\m_3}{\m_2}} \dots \rp^{\frac{\m_{r}}{\m_{r-1}}}\rp,
\end{equation}
where
\begin{eqnarray}\label{eq:exampleseq2}
Z & \triangleq & \sum_{i_1=1}^{l}\lp\sum_{i_2=1}^{l}e^{\beta D_0^{(i_1,i_2)}} \rp^{s} \nonumber \\
 D_0^{(i_1,i_2)} & \triangleq & \sqrt{t}(\y^{(i_2)})^T
 G\x^{(i_1)}+\sqrt{1-t}\|\x^{(i_2)}\|_2 (\y^{(i_2)})^T\lp\sum_{k=1}^{r+1}\u^{(2,k)}\rp\nonumber \\
 & & +\sqrt{t}\|\x^{(i_1)}\|_2\|\y^{(i_2)}\|_2\lp\sum_{k=1}^{r+1}u^{(4,k)}\rp +\sqrt{1-t}\|\y^{(i_2)}\|_2\lp\sum_{k=1}^{r+1}\h^{(k)}\rp^T\x^{(i)},
 \end{eqnarray}
and
${\mathcal X}=\{\x^{(1)},\x^{(2)},\dots,\x^{(l)}\}$ with $\x^{(i)}\in \mR^{n},1\leq i\leq l$ and ${\mathcal Y}=\{\y^{(1)},\y^{(2)},\dots,\y^{(l)}\}$ with $\y^{(i)}\in \mR^{m},1\leq i\leq l$. As \cite{Stojnicnflgscompyx23} suggested, for the concreteness, we assume the so-called linear (proportional growth) regime, $\frac{m}{n}=\alpha$, with $\alpha$ remaining a constant as $n\rightarrow\infty$.

\subsection{Hopfield models}
\label{sec:hop}

Choosing $s=1$, ${\mathcal X}=\{-\frac{1}{\sqrt{n}},\frac{1}{\sqrt{n}}\}^n$, and ${\mathcal Y}=\mS^m$ (where $\mS^m$ is $m$-dimensional unit sphere), we have that
\begin{equation}\label{eq:exampleseq3}
  \lim_{n,\beta\rightarrow\infty}\psi\lp\left \{-\frac{1}{\sqrt{n}},\frac{1}{\sqrt{n}}\right \}^n,\mS^m,\p,\q,\m,\beta,1,1\rp=  \lim_{n\rightarrow\infty} \frac{\max_{\x\in\{-\frac{1}{\sqrt{n}},\frac{1}{\sqrt{n}}\}^n}\|G\x\|_2}{\sqrt{n}}
\end{equation}
 and  $\lim_{n,\beta\rightarrow\infty}\psi(\{-\frac{1}{\sqrt{n}},\frac{1}{\sqrt{n}}\}^n,\mS^m,\p,\q,\m,\beta,1,0)$ are precisely the ground state energy of the  so-called positive square root Hopfield model and its decoupled counterpart in the thermodynamic limit. Various variants of the Hopfield models have been studied over last more than a half of a century. While we defer any further discussion regarding the specialization of the above machinery to the Hopfield models to separate papers, we mention here that their importance is rather extraordinary on a multitude of levels.  Studying both highly theoretical and very practical aspects of these models has been on the forefront of the high level research over the last several decades in the most diverse of the scientific disciplines, including theoretical mathematics and  probability, statistical physics, neural/biological networks, numerical optimization and many others  (more on the historical overview and current state of the art results can be found in, e.g.,  \cite{Hop82,PasFig78,Hebb49,PasShchTir94,ShchTir93,BarGenGueTan10,BarGenGueTan12,Tal98,StojnicMoreSophHopBnds10}).

For the completeness, we also mention that for $s=-1$ one analogously has that
\begin{equation}\label{eq:exampleseq4}
  \lim_{n,\beta\rightarrow\infty}\psi\lp\left \{-\frac{1}{\sqrt{n}},\frac{1}{\sqrt{n}}\right \}^n,\mS^m,\p,\q,\m,\beta,-1,1\rp=   - \lim_{n\rightarrow\infty} \frac{\min_{\x\in\{-\frac{1}{\sqrt{n}},\frac{1}{\sqrt{n}}\}^n}\|G\x\|_2}{\sqrt{n}}
\end{equation}
 and  $\lim_{n,\beta\rightarrow\infty}\psi(\{-\frac{1}{\sqrt{n}},\frac{1}{\sqrt{n}}\}^n,\mS^m,\p,\q,\m,\beta,-1,0)$ are the thermodynamic limit values of the negative square root Hopfield model  ground state energy and its decoupled counterpart.

\subsection{Asymmetric Little models}
\label{sec:litt}

Choosing $s=1$, ${\mathcal X}=\{-\frac{1}{\sqrt{n}},\frac{1}{\sqrt{n}}\}^n$, and ${\mathcal X}={\mathcal Y}=\{-\frac{1}{\sqrt{m}},\frac{1}{\sqrt{m}}\}^m$, we have that
\begin{equation}\label{eq:exampleseq5}
  \lim_{n,\beta\rightarrow\infty}\psi\lp\left \{-\frac{1}{\sqrt{n}},\frac{1}{\sqrt{n}}\right \}^n,\left \{-\frac{1}{\sqrt{m}},\frac{1}{\sqrt{m}}\right \}^n,\p,\q,\m,\beta,1,1\rp=  \lim_{n\rightarrow\infty} \frac{\max_{\x\in\{-\frac{1}{\sqrt{n}},\frac{1}{\sqrt{n}}\}^n}\|G\x\|_1}{\sqrt{nm}}
\end{equation}
 and  $\lim_{n,\beta\rightarrow\infty}\psi(\{-\frac{1}{\sqrt{n}},\frac{1}{\sqrt{n}}\}^n,\{-\frac{1}{\sqrt{m}},\frac{1}{\sqrt{m}} \}^n,\p,\q,\m,\beta,1,0)$ are precisely the thermodynamic limit values of the positive asymmetric Little model ground state energy and its decoupled counterpart. More on the foundations, relevance, and key state of the art results related to these models can be found in, e.g.,
\cite{BruParRit92,Little74,BarGenGue11bip,CabMarPaoPar88,AmiGutSom85,StojnicAsymmLittBnds11}.

As above, we, for the completeness, for $s=-1$, analogously have that
\begin{equation}\label{eq:exampleseq6}
  \lim_{n,\beta\rightarrow\infty}\psi\lp\left \{-\frac{1}{\sqrt{n}},\frac{1}{\sqrt{n}}\right \}^n,\left \{-\frac{1}{\sqrt{m}},\frac{1}{\sqrt{m}}\right \}^n,\p,\q,\m,\beta,-1,1\rp= - \lim_{n\rightarrow\infty} \frac{\min_{\x\in\{-\frac{1}{\sqrt{n}},\frac{1}{\sqrt{n}}\}^n}\|G\x\|_1}{\sqrt{nm}}
\end{equation}
 and  $\lim_{n,\beta\rightarrow\infty}\psi(\{-\frac{1}{\sqrt{n}},\frac{1}{\sqrt{n}}\}^n,\{-\frac{1}{\sqrt{m}},\frac{1}{\sqrt{m}} \}^n,\p,\q,\m,\beta,-1,0)$ are the thermodynamic limit negative asymmetric Little model ground state energy and its decoupled counterpart.

\subsection{Perceptrons}
\label{sec:perc}

If, in (\ref{eq:exampleseq1}), one chooses $s=-1$ and ${\mathcal Y}=\mS_+^n$ and ${\mathcal X}$ as any subset of $\mS^n$ (where $\mS_+^m$ is the positive orthant portion of the $m$-dimensional unit sphere and $\mS^n$ is the $n$-dimensional unit sphere), then $\psi(\cdot)$  is (a properly adjusted) free energy associated with the classical, so-called, positive perceptron. Many particular versions of such perceptrons are of interest and have been studied throughout the literature over the last several decades. For the concreteness, we below select two, possibly, the most representative ones.

\subsubsection{Spherical perceptron}
\label{sec:sphperc}

If, in addition to choosing  $s=-1$ and ${\mathcal Y}=\mS_+^n$, one also chooses  ${\mathcal X}=\mS^n$  in (\ref{eq:exampleseq1}), then $\psi(\cdot)$  is (a properly adjusted) free energy associated with the classical, so-called, positive \emph{spherical} perceptron (see, e.g., \cite{StojnicGardGen13,StojnicGardSphErr13,StojnicGardSphNeg13,GarDer88,Gar88,SchTir02,SchTir03}). In random optimizations problems (see, e.g., \cite{FPSUZ17,FraHwaUrb19,FraPar16,FraSclUrb19,FraSclUrb20,AlaSel20,StojnicGardGen13,StojnicGardSphErr13,StojnicGardSphNeg13,GarDer88,Gar88,Schlafli,Cover65,Winder,Winder61,Wendel62,Cameron60,Joseph60,BalVen87,Ven86,SchTir02,SchTir03}), one is often interested in its thermodynamic limit ground state value
\begin{equation}\label{eq:exampleseq7}
  \lim_{n,\beta\rightarrow\infty}\psi\lp \psi(\mS^m,\mS_+^m,\p,\q,\m,\beta,1,1\rp=  \lim_{n\rightarrow\infty} \frac{\min_{\x\in\mS^m}\max_{\y\in\mS_+^m} \y^TG\x}{\sqrt{n}}.
\end{equation}
Clearly, the interpolating alternative, $\lim_{n,\beta\rightarrow\infty}\psi\lp \psi(\mS^m,\mS_+^m,\p,\q,\m,\beta,1,0\rp$, is then precisely its decoupled counterpart.

\subsubsection{Binary perceptron}
\label{sec:binperc}

If, in addition to choosing  $s=-1$ and ${\mathcal Y}=\mS_+^n$, one also chooses  ${\mathcal X}=\{-\frac{1}{\sqrt{n}},\frac{1}{\sqrt{n}}\}^n$  in (\ref{eq:exampleseq1}), then $\psi(\cdot)$  is (a properly adjusted) free energy associated with the classical, so-called, positive \emph{binary} perceptron (see, e.g., \cite{StojnicGardGen13,GarDer88,Gar88,StojnicDiscPercp13,KraMez89,GutSte90,KimRoc98}). In random optimizations (see, e.g., \cite{StojnicGardGen13,GarDer88,Gar88,StojnicDiscPercp13,KraMez89,GutSte90,KimRoc98})), one is, again, often interested in
\begin{equation}\label{eq:exampleseq8}
  \lim_{n,\beta\rightarrow\infty}\psi\lp \psi(\left \{-\frac{1}{\sqrt{n}},\frac{1}{\sqrt{n}}\right \}^n,\mS_+^m,\p,\q,\m,\beta,1,1\rp=  \lim_{n\rightarrow\infty} \frac{\min_{\x\in\left \{-\frac{1}{\sqrt{m}},\frac{1}{\sqrt{m}}\right \}^n}\max_{\y\in\mS_+^m} \y^TG\x}{\sqrt{n}}
\end{equation}
 and  $\lim_{n,\beta\rightarrow\infty}\psi\lp \psi( \{-\frac{1}{\sqrt{m}},\frac{1}{\sqrt{m}}\}^n,\mS_+^m,\p,\q,\m,\beta,1,0\rp$, which are precisely the positive binary  perceptron associated ground state energy and its decoupled counterpart in the thermodynamic limit.

The above examples are only a few illustrative ones from a rather unlimited collection (the cited references contain a large number of closely related relevant ones as well). Although small, this set provides a pretty good hint as to how wide could be the range of potential applications of our results. Further studying of the above presented interpolating concepts within the context of each of these applications is therefore of great interest. Such a studying is usually problem specific and we will present many interesting results that can be obtained in these directions in separate papers.

\section{Conclusion}
\label{sec:lev2x3lev2liftconc}

In our companion paper \cite{Stojnicnflgscompyx23}, we introduced a very powerful statistical interpolating/comparison mechanism called \emph{fully lifted} (fl). That concept is a very strong upgrade of the fully bilinear corresponding one presented in \cite{Stojnicgscompyx16}. Here, we present a particular realization of the fl mechanism that relies on a stationarization  along the interpolating path concept. Throughout the process, we uncover a collection of very fundamental relations among  the interpolating parameters. We also show how the whole machinery in particular special cases simplifies to forms readily usable in practice.

As the presented results are extremely powerful  and very generic they can be applied to a plethora of well known scenarios appearing in various random structures and optimization problems. Such applications follow the generic principles introduced here but also require  problem specific adjustments that we will present in separate papers.

\begin{singlespace}
\bibliographystyle{plain}
\bibliography{nflgscompyxRefs}

\begin{thebibliography}{10}

\bibitem{Adler90}
R.~J. Adler.
\newblock {\em An introduction to Continuity, Extrema, and Related Topic for
  General Gaussian Processes}.
\newblock Institute of Mathematical Statistics, 1990.

\bibitem{AlaSel20}
A.~E. Alaoui and M.~Sellke.
\newblock Algorithmic pure states for the negative spherical perceptron.
\newblock 2020.
\newblock available online at \bl{\url{http://arxiv.org/abs/2010.15811}}.

\bibitem{AmiGutSom85}
D.~J. Amit, H.~Gutfreund, and H.~Sompolinsky.
\newblock {\em Phys. Rev. A}, 32:1007, 1985.

\bibitem{BalVen87}
P.~Baldi and S.~Venkatesh.
\newblock Number od stable points for spin-glasses and neural networks of
  higher orders.
\newblock {\em Phys. Rev. Letters}, 58(9):913--916, Mar. 1987.

\bibitem{BarGenGueTan10}
A.~Barra, G.~Genovese, and F.~Guerra.
\newblock The replica symmetric approximation of the analogical neural network.
\newblock {\em J. Stat. Physics}, July 2010.

\bibitem{BarGenGue11bip}
A.~Barra, G.~Genovese, and F.~Guerra.
\newblock Equlibrium statistical mechanics of bipartite spin systems.
\newblock {\em Journal of Physics A: Mathematical and Theoeretical},
  44(245002), 2011.

\bibitem{BarGenGueTan12}
A.~Barra, G.~Genovese, F.~Guerra, and D.~Tantari.
\newblock How glassy are neural networks.
\newblock {\em J. Stat. Mechanics: Thery and Experiment}, July 2012.

\bibitem{BraLieb76}
H.~J. Brascamp and E.~H. Lieb.
\newblock On extensions of the {B}runn-{M}inkowski and {P}rekopa-{L}eindler
  theorems, including inequalities for log concave functions, and with an
  application to the diffusion equation.
\newblock {\em J. Functional Analysis}, 22(4):366--389, 1976.

\bibitem{BruParRit92}
R.~Brunetti, G.~Parisi, and F.~Ritort.
\newblock Asymmetric {L}ittle spin glas model.
\newblock {\em Physical Review B}, 46(9), September 1992.

\bibitem{CabMarPaoPar88}
S.~Cabasino, E.~Marinari, P.~Paolucci, and G.~Parisi.
\newblock Eigenstates and limit cycles in the {SK} model.
\newblock {\em Journal of Physics A: Mathematical and General}, 21:4201, 1988.

\bibitem{Cameron60}
S.~H. Cameron.
\newblock Tech-report 60-600.
\newblock {\em Proceedings of the bionics symposium}, pages 197--212, 1960.
\newblock Wright air development division, {D}ayton, {O}hio.

\bibitem{Cover65}
T.~Cover.
\newblock Geomretrical and statistical properties of systems of linear
  inequalities with applications in pattern recognition.
\newblock {\em IEEE Transactions on Electronic Computers}, (EC-14):326--334,
  1965.

\bibitem{Fernique74}
X.~Fernique.
\newblock Des resultats nouveaux sur les processus {G}aussiens.
\newblock {\em C.R. Acad. Sci. Paris Ser A-B}, 278:A363--A365, 1974.

\bibitem{Fernique75}
X.~Fernique.
\newblock Regularite des trajectoires des fonctions aleatoires {G}aussiens.
\newblock {\em Springer Lecture notes}, 480:1--96, 1975.

\bibitem{FraHwaUrb19}
S.~Franz, S.~Hwang, and P.~Urbani.
\newblock Jamming in multilayer supervised learning models.
\newblock {\em Phys. Rev. Lett.}, 123(16):160602, 2019.

\bibitem{FraPar16}
S.~Franz and G.~Parisi.
\newblock The simplest model of jamming.
\newblock {\em Journal of Physics A: Mathematical and Theoretical},
  49(14):145001, 2016.

\bibitem{FPSUZ17}
S.~Franz, G.~Parisi, M.~Sevelev, P.~Urbani, and F.~Zamponi.
\newblock Universality of the {SAT-UNSAT} (jamming) threshold in non-convex
  continuous constraint satisfaction problems.
\newblock {\em SciPost Physics}, 2:019, 2017.

\bibitem{FraSclUrb19}
S.~Franz, A.~Sclocchi, and P.~Urbani.
\newblock Critical jammed phase of the linear perceptron.
\newblock {\em Phys. Rev. Lett.}, 123(11):115702, 2019.

\bibitem{FraSclUrb20}
S.~Franz, A.~Sclocchi, and P.~Urbani.
\newblock Surfing on minima of isostatic landscapes: avalanches and unjamming
  transition.
\newblock {\em SciPost Physics}, 9:12, 2020.

\bibitem{Gar88}
E.~Gardner.
\newblock The space of interactions in neural networks models.
\newblock {\em J. Phys. A: Math. Gen.}, 21:257--270, 1988.

\bibitem{GarDer88}
E.~Gardner and B.~Derrida.
\newblock Optimal storage properties of neural networks models.
\newblock {\em J. Phys. A: Math. Gen.}, 21:271--284, 1988.

\bibitem{Gordon85}
Y.~Gordon.
\newblock Some inequalities for {G}aussian processes and applications.
\newblock {\em Israel Journal of Mathematics}, 50(4):265--289, 1985.

\bibitem{Guerra03}
F.~Guerra.
\newblock Broken replica symmetry bounds in the mean field spin glass model.
\newblock {\em Comm. Math. Physics}, 233:1--12, 2003.

\bibitem{GutSte90}
H.~Gutfreund and Y.~Stein.
\newblock Capacity of neural networks with discrete synaptic couplings.
\newblock {\em J. Physics A: Math. Gen}, 23:2613, 1990.

\bibitem{HadOhm56}
H.~Hadwiger and D.~Ohmann.
\newblock {B}runn-{M}inkowskischer {S}atz und {I}soperimetrie.
\newblock {\em Math. Z.}, 66:1--8, 1956.

\bibitem{Hebb49}
D.~O. Hebb.
\newblock Organization of behavior.
\newblock {\em New York: Wiley}, 1949.

\bibitem{Hop82}
J.~J. Hopfield.
\newblock Neural networks and physical systems with emergent collective
  computational abilities.
\newblock {\em Proc. Nat. Acad. Science}, 79:2554, 1982.

\bibitem{Joseph60}
R.~D. Joseph.
\newblock The number of orthants in $n$-space instersected by an
  $s$-dimensional subspace.
\newblock {\em Tech. memo 8, project {PARA}}, 1960.
\newblock Cornel aeronautical lab., Buffalo, N.Y.

\bibitem{Kahane86}
J.~P. Kahane.
\newblock Une inegualite du type de {S}lepian et {G}ordon sur les processus
  {G}aussiens.
\newblock {\em Israel Journal of Mathematics}, 55(1):109--110, 1986.

\bibitem{KimRoc98}
J.~H. Kim and J.~R. Roche.
\newblock Covering cubes by random half cubes with applications to biniary
  neural networks.
\newblock {\em Journal of Computer and System Sciences}, 56:223--252, 1998.

\bibitem{KraMez89}
W.~Krauth and M.~Mezard.
\newblock Storage capacity of memory networks with binary couplings.
\newblock {\em J. Phys. France}, 50:3057--3066, 1989.

\bibitem{LedTal91}
M.~Ledoux and M.~Talagrand.
\newblock {\em Probability in Banach spaces: Isopermetry and Processes}.
\newblock Springer (New York), 1991.

\bibitem{Lein71}
L.~Leindler.
\newblock On a certain converse of {H}older's inequality.
\newblock {\em Linear operators and approximation (Proc. Conf., Oberwolfach,
  1971)}, pages 182--184, 1971.
\newblock Internat. Ser. Numer. Math., Vol. 20, 1972.

\bibitem{Lifshits85}
M.~A. Lifshits.
\newblock {\em Gaussian random functions}.
\newblock Kluwer, Boston, 1995.

\bibitem{Little74}
W.~A. Little.
\newblock The existence of persistent states in the brain.
\newblock {\em Math. Biosci.}, 19(1-2):101--120, 1974.

\bibitem{Lus35}
L.~Lusternik.
\newblock Die {B}runn-{M}inkowskische ungleichung fur beliebige messbare
  mengen.
\newblock {\em C. R. Acad. Sci. URSS}, 8:55--58, 1935.

\bibitem{Winder}
R.~O.Winder.
\newblock {\em Threshold logic}.
\newblock Ph. D. dissertation, Princetoin University, 1962.

\bibitem{Pan10}
D.~Panchenko.
\newblock A connection between the {G}hirlanda-{G}uerra identities and
  ultrametricity.
\newblock {\em The Annals of Probability}, 38(1):327--347, 2010.

\bibitem{Pan10a}
D.~Panchenko.
\newblock The {G}hirlanda-{G}uerra identities for mixed p-spin model.
\newblock {\em Comptes Rendus Mathematique}, 348(3-4):189--192, 2010.

\bibitem{Pan13}
D.~Panchenko.
\newblock The {P}arisi ultrametricity conjecture.
\newblock {\em Ann. Math.}, 77(1):383--393, 2013.

\bibitem{Pan13a}
D.~Panchenko.
\newblock {\em The {S}herrington-{K}irkpatrick model}.
\newblock Springer Science {\&} Business Media, 2013.

\bibitem{PasFig78}
L.~Pastur and A.~Figotin.
\newblock On the theory of disordered spin systems.
\newblock {\em Theory Math. Phys.}, 35(403-414), 1978.

\bibitem{PasShchTir94}
L.~Pastur, M.~Shcherbina, and B.~Tirozzi.
\newblock The replica-symmetric solution without the replica trick for the
  {H}opfield model.
\newblock {\em Journal of Statistical Physics}, 74(5/6), 1994.

\bibitem{Prek71}
A.~Prekopa.
\newblock Logarithmic concave measures with application to stochastic
  programming.
\newblock {\em Acta Sci. Math. (Szeged)}, 32:301--316, 1971.

\bibitem{Prek73}
A.~Prekopa.
\newblock On logarithmic concave measures and functions.
\newblock {\em Acta Sci. Math. (Szeged)}, 34:335--343, 1973.

\bibitem{Schlafli}
L.~Schlafli.
\newblock {\em Gesammelte Mathematische AbhandLungen I}.
\newblock Basel, Switzerland: Verlag Birkhauser, 1950.

\bibitem{ShchTir93}
M.~Shcherbina and B.~Tirozzi.
\newblock The free energy of a class of {H}opfield models.
\newblock {\em Journal of Statistical Physics}, 72(1/2), 1993.

\bibitem{SchTir02}
M.~Shcherbina and B.~Tirozzi.
\newblock On the volume of the intrersection of a sphere with random half
  spaces.
\newblock {\em C. R. Acad. Sci. Paris. Ser I}, (334):803--806, 2002.

\bibitem{SchTir03}
M.~Shcherbina and B.~Tirozzi.
\newblock Rigorous solution of the {G}ardner problem.
\newblock {\em Comm. on Math. Physics}, (234):383--422, 2003.

\bibitem{Slep62}
D.~Slepian.
\newblock The one sided barier problem for {G}aussian noise.
\newblock {\em Bell System Tech. Journal}, 41:463--501, 1962.

\bibitem{StojnicGenLasso10}
M.~Stojnic.
\newblock A framework for perfromance characterization of \emph{LASSO}
  algortihms.
\newblock available online at \bl{\url{http://arxiv.org/abs/1303.7291}}.

\bibitem{StojnicUpper10}
M.~Stojnic.
\newblock Upper-bounding $\ell_1$-optimization weak thresholds.
\newblock available online at \bl{\url{http://arxiv.org/abs/1303.7289}}.

\bibitem{StojnicCSetam09}
M.~Stojnic.
\newblock Various thresholds for $\ell_1$-optimization in compressed sensing.
\newblock available online at \bl{\url{http://arxiv.org/abs/0907.3666}}.

\bibitem{StojnicICASSP10block}
M.~Stojnic.
\newblock Block-length dependent thresholds for $\ell_2/\ell_1$-optimization in
  block-sparse compressed sensing.
\newblock {\em ICASSP, IEEE International Conference on Acoustics, Signal and
  Speech Processing}, pages 3918--3921, 14-19 March 2010.
\newblock Dallas, TX.

\bibitem{StojnicICASSP10var}
M.~Stojnic.
\newblock $\ell_1$ optimization and its various thresholds in compressed
  sensing.
\newblock {\em ICASSP, IEEE International Conference on Acoustics, Signal and
  Speech Processing}, pages 3910--3913, 14-19 March 2010.
\newblock Dallas, TX.

\bibitem{StojnicISIT2010binary}
M.~Stojnic.
\newblock Recovery thresholds for $\ell_1$ optimization in binary compressed
  sensing.
\newblock {\em ISIT, IEEE International Symposium on Information Theory}, pages
  1593 -- 1597, 13-18 June 2010.
\newblock Austin, TX.

\bibitem{StojnicICASSP10knownsupp}
M.~Stojnic.
\newblock Towards improving $\ell_1$ optimization in compressed sensing.
\newblock {\em ICASSP, IEEE International Conference on Acoustics, Signal and
  Speech Processing}, pages 3938--3941, 14-19 March 2010.
\newblock Dallas, TX.

\bibitem{StojnicGardGen13}
M.~Stojnic.
\newblock Another look at the {G}ardner problem.
\newblock 2013.
\newblock available online at \bl{\url{http://arxiv.org/abs/1306.3979}}.

\bibitem{StojnicAsymmLittBnds11}
M.~Stojnic.
\newblock Asymmetric {L}ittle model and its ground state energies.
\newblock 2013.
\newblock available online at \bl{\url{http://arxiv.org/abs/1306.3978}}.

\bibitem{StojnicRicBnds13}
M.~Stojnic.
\newblock Bounds on restricted isometry constants of random matrices.
\newblock 2013.
\newblock available online at \bl{\url{http://arxiv.org/abs/1306.3779}}.

\bibitem{StojnicDiscPercp13}
M.~Stojnic.
\newblock Discrete perceptrons.
\newblock 2013.
\newblock available online at \bl{\url{http://arxiv.org/abs/1303.4375}}.

\bibitem{StojnicLiftStrSec13}
M.~Stojnic.
\newblock Lifting $\ell_1$-optimization strong and sectional thresholds.
\newblock 2013.
\newblock available online at \bl{\url{http://arxiv.org/abs/1306.3770}}.

\bibitem{StojnicMoreSophHopBnds10}
M.~Stojnic.
\newblock Lifting/lowering {H}opfield models ground state energies.
\newblock 2013.
\newblock available online at \bl{\url{http://arxiv.org/abs/1306.3975}}.

\bibitem{StojnicGorEx10}
M.~Stojnic.
\newblock Meshes that trap random subspaces.
\newblock 2013.
\newblock available online at \bl{\url{http://arxiv.org/abs/1304.0003}}.

\bibitem{StojnicGardSphNeg13}
M.~Stojnic.
\newblock Negative spherical perceptron.
\newblock 2013.
\newblock available online at \bl{\url{http://arxiv.org/abs/1306.3980}}.

\bibitem{StojnicRegRndDlt10}
M.~Stojnic.
\newblock Regularly random duality.
\newblock 2013.
\newblock available online at \bl{\url{http://arxiv.org/abs/1303.7295}}.

\bibitem{StojnicGardSphErr13}
M.~Stojnic.
\newblock Spherical perceptron as a storage memory with limited errors.
\newblock 2013.
\newblock available online at \bl{\url{http://arxiv.org/abs/1306.3809}}.

\bibitem{Stojnicgscompyx16}
M.~Stojnic.
\newblock Fully bilinear generic and lifted random processes comparisons.
\newblock 2016.
\newblock available online at \bl{\url{http://arxiv.org/abs/1612.08516}}.

\bibitem{Stojnicgscomp16}
M.~Stojnic.
\newblock Generic and lifted probabilistic comparisons -- max replaces minmax.
\newblock 2016.
\newblock available online at \bl{\url{http://arxiv.org/abs/1612.08506}}.

\bibitem{Stojnicnflgscompyx23}
M.~Stojnic.
\newblock Fully lifted interpolating comparisons of bilinearly indexed random
  processes.
\newblock 2023.
\newblock available online at arxiv.

\bibitem{Sudakov71}
V.~N. Sudakov.
\newblock Gaussian random processes and measures of solid angles in {H}ilbert
  space.
\newblock {\em Soviet Math. Dokl.}, 12(1):412--415, 1971.

\bibitem{Tal98}
M.~Talagrand.
\newblock Rigorous results for the {H}opfield model with many patterns.
\newblock {\em Prob. theory and related fields}, 110:177--276, 1998.

\bibitem{Tal05}
M.~Talagrand.
\newblock {\em The Generic Chaining}.
\newblock Springer-Verlag, 2005.

\bibitem{Tal06}
M.~Talagrand.
\newblock The {P}arisi formula.
\newblock {\em Annals of mathematics}, 163(2):221--263, 2006.

\bibitem{Ven86}
S.~Venkatesh.
\newblock Epsilon capacity of neural networks.
\newblock {\em Proc. Conf. on Neural Networks for Computing, Snowbird, UT},
  1986.

\bibitem{Wendel62}
J.~G. Wendel.
\newblock A problem in geometric probability.
\newblock {\em Mathematica Scandinavica}, 1:109--111, 1962.

\bibitem{Winder61}
R.~O. Winder.
\newblock Single stage threshold logic.
\newblock {\em Switching circuit theory and logical design}, pages 321--332,
  Sep. 1961.
\newblock AIEE Special publications S-134.

\end{thebibliography}
\end{singlespace}

\end{document}